\documentclass[a4paper,12pt]{amsart}

\usepackage{amsmath,amsfonts,amsthm,amssymb,dsfont}
\usepackage[alphabetic]{amsrefs}
\usepackage{microtype}

\usepackage{times}

\usepackage{enumerate}

\usepackage{mathrsfs}
\usepackage{enumerate, xspace}
\usepackage{graphicx}
\usepackage{color}

    \usepackage[all, knot]{xy}
    \xyoption{arc} 

\usepackage{verbatim}
\usepackage{wrapfig}
\usepackage{caption}
\usepackage{subcaption}

\usepackage{pgf,tikz}
\usetikzlibrary{arrows}

\usepackage[pdfauthor={Markus Steenbock},pdftitle={Rips-Segev torsion-free groups without the unique product property},pdfkeywords={small cancellation groups, Kaplansky zero-divisor conjecture, unique product groups, random groups,  graphical presentations, Rips-Segev groups, graphical small cancellation},pdftex]{hyperref}

\usepackage{color}

\hypersetup{colorlinks,%
citecolor=black,%
filecolor=black,%
linkcolor=black,%
urlcolor=black,%
}

\DeclareMathOperator{\diam}{diam}
\DeclareMathOperator{\girth}{girth}

\newtheorem{theorem}{Theorem}[section]
\newtheorem{them}{Theorem}
\newtheorem{lemma}{Lemma}[section]
\newtheorem{prop}{Proposition}[section]
\newtheorem{cor}{Corollary}[section]

\newtheorem*{lemma*}{Lemma}
\newtheorem*{cor*}{Corollary}
\newtheorem*{prop*}{Proposition}

\theoremstyle{definition}

\newtheorem*{definition*}{Definition}
\newtheorem{example}{Example}[section]

\theoremstyle{remark}
\newtheorem{remark}{Remark}[section]
\newtheorem*{remark*}{Remark}

\numberwithin{equation}{section}

\begin{document}
\definecolor{qqqqff}{rgb}{0,0,1}

\author{Markus Steenbock}
\address{Universit\"at Wien, Fakult\"at f\"ur Mathematik\\
Oskar-Morgenstern-Platz 1, 1090 Wien, Austria.
}
\email{markus.steenbock@univie.ac.at}

\title{Rips-Segev torsion-free groups without the unique product property} 

\begin{abstract} We generalize the graphical small cancellation theory of Gromov to a graphical small cancellation theory over the free product. We extend Gromov's small cancellation theorem to the free product. We explain and  generalize Rips-Segev's construction of torsion-free groups without the unique product property by viewing these groups as given by graphical small cancellation presentations over the free product. Our graphical small cancellation theorem then provides first examples of Gromov hyperbolic groups without the unique product property. We construct uncountably many non-isomorphic torsion-free groups without the unique product property. We show that the presentations of  generalized Rips-Segev groups are not generic among finite presentations of groups. 
\end{abstract}

\maketitle

\smallskip
\smallskip
\smallskip

The outstanding Kaplansky zero-divisor conjecture states that the group ring of a torsion-free group over an integral domain has no zero-divisors \cites{kaplansky_problems_1957,kaplansky_problems_1970}. The conjecture is still open in general. Unique product groups (or groups with the unique product property, see Definition~p.~\pageref{up}) are torsion-free groups that satisfy the conjecture \cite{cohen_zero_1974}. Examples of unique product groups are abelian and non-abelian free groups, nilpotent groups, and hyperbolic groups which act with large translation length on a hyperbolic space 
\cite{delzant_sur_1997}. 
Answering a question of Passman \cite{passman_algebraic_1977}, Rips and Segev have provided first examples of torsion-free groups without the
 unique product property~\cite{rips_torsion-free_1987}. These groups and their Rips-Segev presentations  have still not been systematically investigated. 
 
 The conjecture is also known for torsion-free elementary amenable groups \cite{kropholler_applications_1988} as well as right-angled Artin groups, right-angled  Coxeter groups and certain finite and elementary amenable extensions of such groups \cite{linnell_strong_2012}. Other examples of torsion-free groups without the unique product property can be found in 
\cite{promislow_simple_1988,carter_new_2013}. 

 Our study of Rips-Segev's groups is motivated by the following two open problems.
\smallskip
\begin{itemize}
\item Do Rips-Segev groups satisfy Kaplansky's zero-divisor conjecture?\\[-10pt]
\item Are unique product groups generic among finitely presented groups?
\end{itemize}
\smallskip

The second question is due to T. Delzant. 
We expect a positive answer to both questions. This would also mean that a generic finitely presented group satisfies Kaplansky's zero-divisor conjecture. 

Rips and Segev defined their group by taking as relators words that are read on the cycles of  a certain finite connected edge-labeled graph. The graph was designed to encode the non-unique product relations. On the other hand, to conclude that their groups are torsion-free and without the unique product property, the authors referred to the small cancellation theory~\cite{rips_torsion-free_1987}*{p. 123}.
However, the statements they refer to either are not applicable (in fact, their presentations do not satisfy 
the classical $C(p)$--condition for minimal sequences as described in~\cite{lyndon_combinatorial_1977}*{Ch. V.8}  as relators have to be of length 4 in a free group) or
their proofs are not present in the literature (the result~\cite{lyndon_combinatorial_1977}*{Th. 9.3, Ch. V}  is available only in a specific case of 
the classical metric $C'(\lambda)$--condition over the free product and not under the $C'(\lambda)$--condition for minimal sequences as required by the  Rips-Segev construction).  In particular, contrary to the classical metric small cancellation conditions, the relators given by Rips-Segev can have long common parts, see our explanation in Section~\ref{rssmallcanc}.  Although Rips-Segev were the first to use a graphical group \emph{presentation}, the small cancellation conditions they consider are not graphical conditions.

We adapt a new viewpoint on the Rips-Segev presentations. Namely,  we show that they satisfy the  graphical  metric small cancellation condition with respect to the free product length 
on the free product of certain torsion-free groups. We prove the following general result, of independent interest, that can then be applied to the Rips-Segev presentations. 

 \begin{them}[cf. Theorems \ref{lii}, \ref{T: tf}, \ref{T: gi}]\label{T: main1}
Let $G_1,\ldots,G_n$ be finitely generated groups.  Let $\Omega$ be a family of finite connected graphs edge-labeled by $G_1\cup \ldots \cup G_n$ so that the graphical metric small cancellation condition with respect to the free product length on the free product $G_1*\cdots *G_n$ is satisfied.
Let $G$ be the group given by the corresponding graphical presentation, that is,  the quotient of $G_1*\cdots *G_n$ subject to the relators being  the words read on the cycles of $\Omega.$ 

Then
$G$ satisfies a linear isoperimetric inequality with respect to the free product length. Moreover, $G$ is torsion-free whenever $G_1,\ldots,G_n$ are torsion-free; 
$G$ is Gromov hyperbolic whenever $G_1,\ldots,G_n$ are Gromov hyperbolic and $\Omega$ is finite.

The graph $\Omega$ injects into the Cayley graph of $G$ with respect to $G_1\cup \ldots \cup G_n$.
\end{them}

Our theorem extends to the free product setting Gromov's graphical small cancellation theorem which states that presentations with the graphical metric small cancellation condition over a free group define torsion-free hyperbolic groups~\cites{gromov_random_2003,ollivier_small_2006}. 
Our result also generalizes to the graphical small cancellation setting a theorem of Pankrat'ev who establishes the hyperbolicity of the free products of hyperbolic groups subject to relators satisfying 
the classical metric small cancellation condition~\cite{pankratev_hyperbolic_1999}.

As a corollary, our theorem provides all details to the Rips-Segev original construction, specifically, we conclude that the Rips-Segev groups are indeed torsion-free.
In addition, this yields the following new result. Note that the reference to the small cancellation conditions for minimal sequences (as stated in~\cite{rips_torsion-free_1987}*{p. 123}) is not sufficient (as it could lead to quadratic Dehn functions instead of linear ones which characterize hyperbolicity in the case of finitely presented groups).   

\begin{them}[cf. Theorem \ref{T: RS are hyperbolic}] \label{gh} The Rips-Segev torsion-free groups without the unique product property are Gromov hyperbolic.\end{them}
 
Hyperbolic groups with large translation length have the unique product property \cite{delzant_sur_1997}. { Hence}, we deduce that the Rips-Segev groups are first examples of torsion-free
hyperbolic groups which possess no action with large translation length on any hyperbolic space. 

We also obtain that the Rips-Segev groups satisfy the Kadison-Kaplansky conjecture (stating that the reduced $C^*$-algebra of a torsion-free group has no non-trivial idempotents)
as, for torsion-free groups, it follows from the Baum-Connes conjecture, proved for Gromov hyperbolic groups~\cite{lafforgue_baum_1998}. 

Extending our graphical viewpoint further, we define the \emph{generalized Rips-Segev presentations} and construct many \emph{new torsion-free groups without the unique product property}. 
Some of our new groups have infinitely many different pairs of subsets without the unique product property. 

It is unknown whether or not there are countably many Rips-Segev groups up to isomorphism. 
We show that elementary Nielsen equivalences and conjugation do not induce isomorphisms of Rips-Segev groups. This allows to produce
a huge family of torsion-free groups  without the unique product property.
\begin{them}[cf. Theorem \ref{infpresgrwup}]
 There are uncountably many non-isomorphic torsion-free groups without the unique product property.  
\end{them}

On the other hand, we prove that all currently known finite presentations of torsion-free groups without the unique product property are not generic among finite presentations of groups
with respect to two different fundamental models of random finitely presented groups.

\begin{them}[cf. Theorem \ref{g1}] Generalized Rips-Segev presentations of torsion-free non-unique product groups are not generic in Gromov's graphical model \cites{gromov_random_2003,ollivier_kazhdan_2007} of finitely presented random groups. 
\end{them}

\begin{them}[cf. Theorem \ref{g2}] Generalized Rips-Segev presentations of torsion-free non-unique product groups are not generic in Arzhantseva-Ol'shanskii's few relators model \cite{arzhantseva_class_1996} of finitely presented random groups. \end{them}

{\bf Acknowledgments.} The author is grateful to Prof. Goulnara Arzhantseva for suggesting the problem, revisions, many valuable remarks and discussions, to his colleagues Christopher Cashen, Dominik Gruber and Alexandre Martin for useful comments, and to the referees for their suggestions.

 The work was supported by the ERC grant ANALYTIC 259527 of \mbox{G.~Arzhantseva} and the University of Vienna.
 M. Steenbock is a recipient of the Vienna University Research Grant 2013 and of the DOC fellowship of the Austrian Academy of Sciences 2014--2015. 
 The results have been announced at the Erwin Schr\"odinger International Institute 
 for
Mathematical Physics 
within the workshop  `Golod-Shafarevich groups and algebras and rank gradient' in August
2012.


\section{Small cancellation theories over the free product}

We first review essentials of classical small cancellation theory. Then we introduce the graphical small cancellation conditions over the free product that we use in our approach to the Rips-Segev groups. 

\subsection{Classical small cancellation theory over the free product}\label{S: classical small cancellation theory}

Let $F=G_1*\cdots *G_d$ be the free product. The groups $G_i$ are called \emph{factors} of $F$. Each non-trivial element $w\in F$ can be represented as a product $w=h_1\cdots h_n$, 
where each $h_i\not=1$ is in one of the factors. If successive $h_i$ and $h_{i+1}$ are not in the same factor, then $h_1\cdots h_n$ is the \emph{normal form} of $w$. 
Each non-trivial element $w\in F$  has a unique representation in normal form. The integer $n$ in such a representation is the \emph{free product length} (or the \emph{syllable length}) of $w$,
denoted by $\lvert w \rvert_*.$ For instance, if $g_i\in G_i$, $g_i^2\not=1$, then $\lvert g_1^2\rvert_*=1$ and $\lvert g_1^2g_2g_1^{-1}\rvert_*=3.$

Let $w\in F$ be non-trivial, $w=h_1\cdots h_n$, and some successive $h_i$ and $h_{i+1}$ are in the same factor. If $h_i=h_{i+1}^{-1}$, we say $h_i$ and $h_{i+1}$ \emph{cancel}. 
If  $h_i\not=h_{i+1}^{-1}$, we say  $h_i$ and $h_{i+1}$  \emph{consolidate} to $a=h_ih_{i+1}$ and give $w=h_1\cdots h_{i-1} a h_{i+2}\cdots h_n$. 
We call $w=h_1\cdots h_n$  \emph{weakly reduced}  if there is no cancellation between successive $h_i$ and $h_{i+1}$, consolidations are allowed.  
We call $w=h_1\cdots h_n$  \emph{weakly cyclically reduced}  
  if $n\leqslant 1$  or for each cyclic permutation $\sigma$ we have $h_{\sigma(1)}\cdots h_{\sigma(n)}$ weakly reduced. 
  { Then} $h_{\sigma(1)}\cdots h_{\sigma(n)}$ is a \emph{weakly cyclically reduced conjugate} of~$w$.
Given $w,w'\in F$, we say that the product $ww'$ is \emph{weakly reduced}
  if $w=h_1\cdots h_n, w'=h'_1\cdots h'_l$ such that $h_1\cdots h_nh'_1\cdots h'_l$ is weakly reduced.

\smallskip
  
Let $R\subseteq F$ and $G:=F/\langle\langle R  \rangle\rangle$ be the quotient of $F$ by the normal subgroup generated by $R$. 
  We say $R$ is \emph{symmetrized} if 
 it contains all weakly  cyclically reduced conjugates of $r$ and $r^{-1}$ for each $r\in R$. 
We always assume $R$ is symmetrized  as both $R$ and its natural symmetrization define the same group~$G$.

\smallskip

An element $p\in F$ is a \emph{piece} if for distinct elements $r_1, r_2\in R$, we have
$r_1=pu_1, r_2=pu_2$ for some $u_1,u_2\in F$ and these products are weakly reduced. 

\smallskip

Let $0<\lambda<1$. A subset $R\subseteq F$ satisfies the \emph{$C'_{*}(\lambda)$--small cancellation condition over $F$} (or, briefly, the \emph{$C'_{*}(\lambda)$--condition}) if for every piece $p$ we have 
\[
\lvert p \rvert_*<\lambda \min \{\lvert r \rvert_*\mid r\in  R \}.
\]
{ Then} the group $G=F/\langle\langle R  \rangle\rangle$ is a \emph{$C'_{*}(\lambda)$--small cancellation group} (or just  a \emph{$C'_{*}(\lambda)$--group}).

Let $\Lambda=\max\{ \lvert p \rvert_*\mid p \text{ is a piece}\}$  be the \emph{maximal piece length} and $\gamma=\min \{\lvert r \rvert_*\mid r\in R\}$ be the \emph{minimal relator length}.
Then the $C'_{*}(\lambda)$--condition states: 
\[
\frac{\Lambda}{\gamma}<\lambda .
\]

\smallskip

 The $C'_*(\lambda)$--condition has a geometric interpretation in the language of van Kampen diagrams. 
 
 \smallskip

A \emph{diagram} $D$ is a finite  oriented planar 2-complex  with a specified  embedding $i$  of $D$ into the plane.
 An \emph{edge} is a 1-cell of $D$. The set of all edges is denoted by $E(D)$. 
  The \emph{boundary} $\partial D$ is the set of edges and vertices on the boundary of $i(D)$ in the plane. 
 A \emph{face} $\Pi$ is a (closed) 2-cell of $D$.  { The interior of every face is homeomorphic to the open disc.}
 The number of all faces is denoted by $\lvert D \rvert$.  We say that a face $\Pi$ is \emph{not simply-connected (in D)} if distinct edges or vertices in $\partial\Pi$ are identified in $D$. Otherwise, $\Pi$ is called \emph{simply-connected (in D)}. 
 
 An edge or a vertex of $D$ is called \emph{inner} if it is not in $\partial D$. { The \emph{exterior boundary} $\partial_{ext}\Pi$ of a face $\Pi$ is the intersection of $\partial \Pi$ with $\partial D$  in $D$. The \emph{inner boundary} $\partial_{int} \Pi$ is the closure of $\partial\Pi -\partial_{ext}\Pi$  in $D$.  A face with non-empty exterior boundary is called an \emph{exterior face}, otherwise it is called an \emph{inner face}.}
 
  A path $p$ of edges in $D$ is \emph{simple}   if no edge $e$ or its inverse $e^{-1}$ occurs more than once in $p$. A path is a \emph{cycle} in $D$ if its starting and terminal vertices coincide. 
 A path is an \emph{arc} in $D$ if all of its vertices, except possibly the endpoints, have degree 2 in $D$.
An \emph{inner segment} of $D$ is an arc made of inner edges. A \emph{boundary cycle} of a face $\Pi$ is a cycle of minimal length including all the edges of $\partial\Pi$.
A \emph{boundary cycle of $D$} is a cycle of minimal length including all the edges of $\partial D$, which does not cross itself. 

\smallskip

  A \emph{labeling} of ${D}$ by $F$ is a map $\omega\colon E(D)\to F$ 
  assigning to each edge $e\in E(D)$ an element $\omega(e)$ of a factor of $F$ so that
  $\omega(e^{-1})=\omega(e)^{-1}$.  The label $\omega(p)$ of a path $p=(e_1,\ldots, e_n)$ is the concatenation $\omega(e_1)\cdots\omega(e_n)$.  The labels of two paths are called \emph{equal (in F)} if they represent the same element in~$F$.

   If  $\omega\colon E(D)\to Y\subseteq F$ for some $Y\subseteq F$, we say that $\omega $ is a labeling by $Y$. Usually we consider $Y=X\sqcup X^{-1}$, where $X:=X_1\sqcup \ldots \sqcup X_d$ and each $X_i$ is a finite set generating the factor $G_i$ of $F$. 
   We denote by $\lvert \, .\, \rvert$ the \emph{word length} on $F$ with respect to $X$.

\begin{definition*}
A \emph{van Kampen diagram}  for an element $w \in F$, over a set of relators $ R \subseteq F, $ is a 
connected and simply-connected 
diagram $D$ labeled by $F$ such that: 
\begin{itemize} \item The label $\omega(e_1)\cdots\omega(e_n)$ of a  boundary cycle $e_1,\ldots,e_n$ of $D$ is weakly reduced and represents~$w$; 
 \item 
 The label $\omega(e_1')\cdots\omega(e_m')$ of a boundary cycle $e_1',\ldots,e_m'$  of each face $\Pi$ in $D$ is weakly cyclically  reduced and represents a relator~$r\in R$.\end{itemize} 
  \end{definition*}

The next result is a variant for the free products of the fundamental van Kampen Lemma.

\begin{lemma}[\cite{lyndon_combinatorial_1977}*{p. 276, Ch. V.9}] An element $w \in F$ belongs to the normal subgroup $\langle\langle R\rangle\rangle$ of $F$ generated by $R\subseteq F$ if and only if there exists a van Kampen diagram for $w$ over $R$. 
\end{lemma}

Let $0<\lambda<1$. {An arc} $s$ in $D$ is said to satisfy the \emph{$C'_*(\lambda)$--small cancellation condition} over $F$ if    
\[
\lvert \omega(s)\rvert_*<\lambda\min\{ \lvert r\rvert_*\mid r= \omega(\partial\Pi), \Pi \text{ is a face in } D\}.
 \]
If all inner segments of $D$ satisfy the $C'_*(\lambda)$--small cancellation condition over F, then $D$ is said to satisfy the \emph{$C'_*(\lambda)$--small cancellation condition over F}
(or briefly, the \emph{$C'_*(\lambda)$--condition}).
  
\smallskip
 
  An inner segment $s$ in $D$ \emph{originates}\label{D: originate in R} in $R$ if $s$ is in the common boundary of two faces $\Pi_1$ and $\Pi_2$ with boundary cycles $sp_1$ and $sp_2$ such that $\omega(p_1)=\omega(p_2)$. Van Kampen diagrams with no  originating inner segments are called \emph{reduced}, see e.g. \cite{lyndon_combinatorial_1977}*{p.241, Ch.V.9}.
   A face $\Pi$ \emph{self-intersects} along $s$, if a boundary cycle of $\Pi$ is $q_1sq_2s^{-1}$, where $q_1$ and $q_2$ are distinct non-trivial cycles in $D$. In this case, $\Pi$ is \emph{not} simply-connected. 
  
  The labels of inner segments that do not originate  in $R$ represent pieces, cf. \cite{lyndon_combinatorial_1977}*{p.277, L.9.2, Ch.V.9}. The label of originating inner segments is not controlled by the small cancellation condition on $R$. Therefore, a priori, a van Kampen diagram $D$ over $R$ does not satisfy the $C'_*(\lambda)$--condition. Moreover, $D$ can have not simply-connected faces. 
  If $D$ is \emph{minimal}, that is, $|D|$ is minimal among all van Kampen diagrams for $w$, then the following variant of the Greendlinger lemma, cf. \cite{lyndon_combinatorial_1977}*{p. 278, Th. 9.3, Ch. V.9}, ensures 
  both the $C'_*(\lambda)$--condition and the simply-connectedness of all faces in $D$.

\begin{theorem}[Classical small cancellation lemma]\label{T: classical small cancellation lemma}  Let $0<\lambda\leqslant 1/6.$  Let $D$ be a  labeled simply-connected diagram with simply-connected faces that satisfies the $C'_*(\lambda)$--small cancellation condition over the free product. 
\begin{itemize}
  \item If $D$ has more than two faces then there are at least two exterior faces $\Pi$ such that 
   $$\lvert\omega(\partial_{ext}\Pi)\rvert_*>\left( 1-{3}{\lambda}\right){\lvert\omega(\partial\Pi)\rvert_*},$$ 
  $\partial_{int}\Pi$ consists of at most three inner segments, and $\partial_{ext}\Pi$ is connected. 
  \item The following inequality is satisfied $$\lvert\omega(\partial D)\rvert_*> (1-6\lambda) \sum_{\Pi_i\text{ is a face in }D} \lvert\omega(\partial \Pi_i)\rvert_*.$$ 
  \item The label of $\partial D$ is at least as long as (with respect to the length function $\lvert\, .\,\rvert_*$) the label of $\partial\Pi$, for each face $\Pi$ in $D.$
 \end{itemize}
\end{theorem}

\begin{cor}[\cite{lyndon_combinatorial_1977}*{p. 277, proof of L. 9.2, Ch. V}]\label{C: Ollivier} Let $0<\lambda\leqslant 1/6$. Let $R\subseteq F$ satisfy the  { $C'_*(\lambda)$--condition} over $F$. Then every minimal van Kampen diagram satisfies the { $C'_*(\lambda)$--condition} over $F$ and all its faces are simply-connected. 
\end{cor}
{ In Section~\ref{S: van Kampen diagrams}, we extend our proof below to the graphical small cancellation setting.}

\begin{proof} We apply Theorem \ref{T: classical small cancellation lemma}. Let $D$ be a minimal van Kampen diagram for an element $w\in F$ over $R$.  
The faces of $D$ correspond to relators $r_1,\ldots, r_n\in R$. The minimality condition on $D$ ensures that $n$ is the minimal number
such that $w$ can be expressed as a product of conjugates of relators.

Let $s$ be an inner segment in $D$ that originates in $R$. We delete $s$ and obtain a new diagram over $R$ with a face whose labels of the boundary cycles represent the identity in $F$.  Thus, 
$w$ can be expressed as a product of conjugates of $n-2$ relators. This contradicts to the minimality of $n$.

Suppose $D$ contains a face $\Pi$ which self-intersects along an inner segment $s$ in $D$: a boundary cycle of $\Pi$ is $q_1sq_2s^{-1}$, where $q_1$ and $q_2$ are distinct non-trivial cycles in $D$. 
Take such an \emph{innermost} $\Pi$: the subdiagrams $K_1$ and $K_2$ of $D$ that fill the cycles $q_1$ and $q_2$ are simply-connected diagrams and all their faces, besides $\Pi$, are simply-connected. Suppose that $K_1$ contains $K_2$.   
The labels on all inner segments  $s'$ in the common boundary of  $\Pi$ and $K_2$ represent pieces in $R$. Theorem \ref{T: classical small cancellation lemma} applies to $K_2$ and yields an inner segment $s''$ in the common boundary of $\Pi$ and $K_2$ whose label is larger than $\frac{1}{2}{\lvert\omega(\partial\Pi)\rvert_*}$. This is a contradiction. 
\end{proof}

A \emph{linear isoperimetric inequality} with respect to the free product length states that there is a constant~$C>0$ such that  all minimal van Kampen diagrams $D$ over $R$ satisfy $\lvert \omega(\partial D) \rvert_* \geqslant C \lvert D \rvert$. 
  A \emph{linear isoperimetric inequality} with respect to the word length metric states that  there is a constant $C>0$ such that  all minimal van Kampen diagrams $D$ satisfy $\lvert\omega(\partial D)\rvert  \geqslant  C \lvert D\rvert $. The Gromov hyperbolicity of $G$ is known to be equivalent to the linear isoperimetric inequality with respect to the word length metric. 
  
  The second inequality of Theorem \ref{T: classical small cancellation lemma} yields a linear isoperimetric inequality  with respect to the free product length. As $\lvert\, .\,\rvert_* \leqslant \lvert\, .\,\rvert$, we obtain a linear isoperimetric inequality with respect to the word length metric. If each factor $G_i$ of $F$ is  Gromov hyperbolic, then one obtains the first of the following well-known small cancellation theorems.

\begin{theorem}\label{C: hyperbolic}\emph{\cite{pankratev_hyperbolic_1999}} Let $G$ be a $C'_*(1/6)$--small cancellation group with relators $R$ over $F$. Then 
 $G$ satisfies a linear isoperimetric inequality with respect to $\lvert\, .\, \rvert_*.$ If $F$ is the free product of finitely many Gromov hyperbolic groups and $R$ is finite, then $G$ is Gromov hyperbolic.
\end{theorem}

The following result is known as the \emph{torsion theorem} for classical small cancellation groups.

\begin{theorem}\label{C: t.f.}\emph{\cite{lyndon_combinatorial_1977}*{p. 281, Th. 10.1, Ch. V.}} Let $G$ be a $C'_*(1/8)$--small cancellation group with relators $R$  over $F$. If $F$ is torsion-free and no relator in $R$ is a  proper power, then $G$ is torsion-free.  
\end{theorem}

For each element $r\in R$, take a  path $p_r$ labeled by $F$, whose label is the normal form of $r$.
Let $c_r$ be the \emph{relator cycle}, a cycle graph obtained from $p_r$ by identifying the endpoints.  The following is an easy consequence of the first statement of Theorem~\ref{T: classical small cancellation lemma}.

\begin{theorem}\label{C: injection}
 Let $G$ be a $C'_*(1/6)$--small cancellation group  with relators $R$ over $F$. Then for each $r\in R$ the relator cycle $c_r$ injects into the Cayley graph of $G$ with respect to $G_1\cup\ldots\cup G_d$.
\end{theorem}

\subsection{Graphical small cancellation theory over the free product} \label{S: graphical}

Let $\Omega$ be an oriented graph (finite or infinite). It is a 2-complex with no faces, so the above definitions apply to graphs. In particular, a \emph{labeling} of $\Omega$  by $F$ is a map $\omega\colon E(\Omega)\to F$ such that { the label $\omega(e)$ of every edge $e$ is in a factor of $F$ and }  $\omega(e^{-1})=\omega(e)^{-1}$. 
If  $\omega\colon E(\Omega)\to Y\subseteq F$ for $Y\subseteq F$, we say that $\omega $ is a labeling by~$Y$. 

 Let $\Omega$ be labeled by $X\sqcup X^{-1}$. 
The labeling of $\Omega$, or the labeled graph $\Omega$ itself, is called \emph{reduced (over F)} if the label of  every  non-trivial simple path in $\Omega$ represents a non-trivial element in $F$. 
 
In particular, in a reduced labeled $\Omega$ the labels of consecutive edges do not cancel. Therefore, 
the label of each non-trivial simple path is weakly reduced and the label of each non-trivial simple cycle is weakly cyclically reduced. 
\begin{definition*}\label{D: graphical presentation} Let $\mathcal{C}$ denote a set of cycles generating the fundamental group of $\Omega$. Let $R\subseteq F$ denote the subset of elements represented by the labels of the cycles in $\mathcal{C}$. The quotient group $G(\Omega):=F/\langle\langle R\rangle \rangle$ is the \emph{graphically presented group} with relators $R$ over $F$.
\end{definition*}
  The group $G(\Omega)$ does not depend on the choice of $\mathcal{C}$ (e.g. we can always assume that $\mathcal{C}$ is made of simple cycles). If $\Omega$ is finite we can choose such a finite $\mathcal{C}$. Then  $G(\Omega)$ is finitely presented whenever $F$ is a finitely presented group.  

  \smallskip
  
Our next aim is to introduce the graphical small cancellation conditions over the free product~$F$. We start with the notion of a piece, then we focus on non-trivial relations in $F$.

An \emph{immersion} of labeled graphs is a locally injective graph morphism which preserves the labellings. \label{D: graphical piece}
A \emph{(graphical) piece} in $\Omega$ is a labeled path $p$ immersed in $\Omega$ such that there is a labeled path $q$ immersed in $\Omega$, the immersions of $p$ and $q$  are distinct, and  the labels satisfy $\omega(q)=\omega(p)$  in $F$.

\smallskip

In our specific applications below, we consider the following two cases.

 1.  Each factor $G_i$ of $F$ is a free group on $X_i$ and $\Omega$ is labeled by $X\sqcup X^{-1}$. The labeling of $\Omega$ is reduced if the word read on two subsequent edges does not cancel to the empty word. 
  \smallskip

  2. $F=G_1*G_2$ and the factors $G_1$ and $G_2$ are torsion-free. Let $1\not=a\in G_1$, $1\not=b\in G_2$. A labeling of $\Omega$ by $\{a^{\pm1},b^{\pm1}\}$ is reduced if the label on two subsequent edges does not cancel to the empty word. 

\smallskip

In both cases, for a labeled path $p$ to be a piece, it suffices that $p$ has at least two distinct immersions into $\Omega$. In Case 1, our definition of graphical pieces
coincides with that given in \cite{ollivier_small_2006}*{p. 76}.  In Case 2, the elements $a$ and $b$ generate a free subgroup of $G_1*G_2$. Thus, Case 2 is a generalization of Case 1. 

\medskip

{

 We now introduce certain operations on a labeled graph $\Omega$ that preserve the group $G(\Omega)$. 
We say two graphs $\Omega'$, labeled by $X'\sqcup {X'}^{-1} $, and $\Omega$, labeled by $X\sqcup X^{-1}$, are \emph{equivalent  (over F)}, if one is obtained from the other by means of finitely many combinations of the following three \emph{graph transformations}. 

 \smallskip
 
 1. \emph{AO-move over the free product}. An \emph{arc}  is a path in $\Omega$ all of whose vertices, except possibly the endpoints, have degree 2.
 Let $p$ be an arc of $n$ edges in $\Omega$ with a weakly reduced label $x_{i_1}^{\epsilon_{i_1}}\cdots x_{i_n}^{\epsilon_{i_n}}$, $\epsilon_{i_1}, \ldots ,\epsilon_{i_n} \in \{\pm 1\}$, $x_{i_1},\ldots,x_{i_n}\in X_i$.  Suppose ${x'_{j_1}}^{\epsilon_{j_1}}\cdots {x'_{j_m}}^{\epsilon_{j_m}}$ , $\epsilon_{j_1}, \ldots ,\epsilon_{j_m} \in \{\pm 1\}$, ${x'_{j_1}},\ldots,{x'_{j_m}}\in X_i',$ is weakly reduced and $x_{i_1}^{\epsilon_{i_1}}\cdots x_{i_n}^{\epsilon_{i_n}}={x'_{j_1}}^{\epsilon_{j_1}}\cdots {x'_{j_m}}^{\epsilon_{j_m}}$ in~$F$. Let $q$ be a labeled path of $m$ edges whose label is ${x'_{j_1}}^{\epsilon_{j_1}}\cdots {x'_{j_m}}^{\epsilon_{j_m}}$. We  \emph{replace $p$ with $q$}: we identify the starting vertex of $q$ with the starting vertex of $p$ in $\Omega$, and the terminal vertex of $q$ with the terminal vertex of $p$. Then we delete all the edges of $p$ together with all of its vertices, except the endpoints.

 \smallskip
 
 2. \emph{Reduction over the free product}.  {Let $p$ be a simple path in $\Omega$ whose label represents the identity in~$F$.} We 
 identify the starting vertex and the terminal vertex of $p$ and delete the terminal edge. 

\smallskip

3. \emph{Deletion}. We delete edges incident to vertices of degree 1 together with these vertices.

}
\medskip

If two graphs  $\Omega'$ and $\Omega$ are equivalent over $F$, then they define the same group: $$G(\Omega')=G(\Omega).$$

The AO-moves over a given group $H$ were first defined and applied by Arzhantseva and Ol'shanskii in \cite{arzhantseva_class_1996}*{p. 351f} to  Stallings graphs. If two such graphs  are equivalent, then they define the same subgroup of $H$ \cite{arzhantseva_class_1996}*{p. 352, L. 1}. The terminology `AO-move' is due to \cite{kapovich_2005_genericity}*{p. 9, Def. 2.7}.

In Case 1,  if $e_1,e_2$ is a simple path such that $\omega(e_1)\omega(e_2)$ cancels to the identity in the free group, then the Reduction of $e_1, e_2$ is known as the \emph{folding} of $e_1$ and $e_2$.

\medskip

\begin{definition*} The \emph{ maximal piece length}  in a labeled graph $\Omega$ is defined by 
\[ 
\Lambda:=\max \{\lvert \omega(p) \rvert _*\mid p \hbox{ is a piece in  a reduced labeled graph $\Omega'$ equivalent to } \Omega\},
\] and the \emph{minimal cycle length}  in a labeled graph $\Omega$ is defined by
\[
 \gamma:=\min\{\lvert \omega(c) \rvert _* \mid c \text{ is a cycle in  a reduced labeled graph $\Omega'$ equivalent to } \Omega \}.
 \]
 Let $0<\lambda<1.$
 A  
 reduced
labeling of $\Omega$ satisfies the \emph{$Gr_*'(\lambda)$--graphical small cancellation condition}  if $$\frac{\Lambda}{\gamma}<\lambda.$$
 In this case, the  group $G(\Omega)$ is a \emph{$Gr_*'(\lambda)$--graphical small cancellation group} with relators $R$ over $F$. 
\end{definition*}

If $F$ is a free group and we consider our graphical small cancellation condition with respect to the word length metric on $F$, then we obtain Gromov's graphical small cancellation condition over the free group  \cite{ollivier_small_2006}*{p.77, Th.1}.

\smallskip

We now provide a useful criterion on a reduced labeling of $\Omega$ that ensures the {$Gr_*'(\lambda)$--condition}.

\begin{prop}\label{P: criterion}
 Let $\omega$ be a reduced labeling of $\Omega$ and $0<\lambda <1$. If for all graphical pieces $p$ and all non-trivial  labeled cycles $c$ in $\Omega$ we have that 
 \[
 \frac{\lvert  \omega(p) \rvert_* +2}{\omega(c)}<\lambda,
  \]
 then the labeling of $\Omega$ satisfies the $Gr'_*(\lambda)$--graphical small cancellation condition.
\end{prop}

\begin{proof}
There are two types of vertices in $\Omega$. At a \emph{vertex of first type} there are at least two edges such that their labels are in distinct factors of $F$. The remaining vertices are \emph{vertices of second type}. If $p$ is a path immersed in $\Omega$, the vertices of first type in $p$ have two edges of $p$ whose labels are in distinct factors. Hence, the starting and the terminal vertex of $p$ are not included in the vertices of first type in $p$. Then $\lvert \omega(p) \rvert_*-1$ equals the number of vertices of first type in $p$. If $c$ is a reduced labeled cycle, then $\lvert \omega(c) \rvert_*$ equals the number of vertices of first type in $c$.  We observe that vertices of first type in a reduced labeled graph are never identified by applications of {the graph transformations}.
 
   Let $\Omega'$ be { reduced labeled and} equivalent to $\Omega$. The above arguments imply that the minimal length of the non-trivial cycles in $\Omega$ and the minimal length of the non-trivial cycles in $\Omega'$ are equal.
   
     Let  $p'$ be a path immersed in $\Omega'$ which starts and terminates at vertices of first type. Then $p'$  has a preimage $p$ in $\Omega$. The vertices of first type in $p$ are the preimages of the vertices of first type in $p'$. Moreover, the starting and terminal vertex of $p$ are of first type in $\Omega$ and the label of $p$ equals the label of $p'$.
 
 Now suppose that $p'$ is a piece in $\Omega'$. Let $q'$ be a second path whose immersion in $\Omega'$ is distinct from the immersion of $p'$ and whose label equals the label of $p'$. If the starting and terminal vertices of both $p'$ and $q'$ are of first type in $\Omega'$, then there are preimages $p$ and $q$ in $\Omega$ as above. The immersions of $p$ and $q$ are distinct in $\Omega$ as vertices of first type cannot be identified by applications of AO-moves and Reductions. We conclude that $p$ is a piece in $\Omega$. The labels of $p$ and $p'$ are equal, so we obtain the required inequality. Let us now consider the case that $p'$ or $q'$ does not start or terminate at vertices of first type. There are three possibilities for such a piece $p'$. Either $p'$ and $q'$ have no vertex of first type, $p'$ and $q'$ have one vertex of first type, or $p'$ and $q'$ contain at least two vertices of first type. In the first two cases, we have that $\lvert \omega(q) \rvert_* \leqslant 2$. Hence, the label of 
$p'
$ is shorter than $\lambda\gamma$. In the third case, $p'$ and $q'$ have a subpath $\tilde{p}$ and $\tilde{q}$ whose starting and terminal 
vertex are of first type in $\Omega'$ and whose labels satisfy $\omega(\tilde{p})=\omega(\tilde{q})$ in $F$. We choose $\tilde{p}$ and $\tilde{q}$ maximal so that there are no subpaths of $p'$ and $q'$ whose starting and terminal vertices are of first type in $\Omega'$ and which contain $\tilde{p}$ or $\tilde{q}$ with the above properties. This implies that $\lvert \omega(q') \rvert_*\leqslant \lvert \omega(\tilde{q}) \rvert_* +2$. By the above arguments, $\tilde{q}<\lambda \gamma$. We conclude that $\lvert \omega(q) \rvert_*<\lambda \gamma$. This completes the proof.
\end{proof}

This criterion is later used to conclude that  specific examples of labellings of graphs satisfy the graphical small cancellation condition over the free product.  We proceed by studying van Kampen diagrams over graphical group presentations. This is the first step in proving our graphical small cancellation theorem, Theorem~\ref{T: main1}.

\subsection{Van Kampen diagrams}\label{S: van Kampen diagrams}

In this section, we study van Kampen diagrams over relators $R\subseteq F$ given by the labels of cycles of~$\Omega$. These van Kampen diagrams are different to those studied in Section \ref{S: classical small cancellation theory}. Namely, they have new `originating' inner segments that are not controlled by the small cancellation condition. We `delete' such inner segments to obtain diagrams that satisfy the $C'_*(\lambda)$--condition. 

\smallskip

Let $\Omega$ be a reduced labeled graph whose labeling satisfies the $Gr_*'(\lambda)$--condition for $0<\lambda\leqslant 1/6$. 
 Let $\widetilde{\Omega}$ be the 2-complex obtained by filling all reduced labeled cycles of $\Omega$ with a disc. Let $D$  be a van Kampen diagram over $R$. Let $\Pi$ be a face of $D$. The label on a boundary cycle $c$ of $\Pi$ equals in $F$ to a relator $r\in R$. In $\widetilde{\Omega}$, we find a 2-cell ${\Pi}'$ {, called a \emph{lift} of $\Pi$ in $\widetilde{\Omega}$, with a boundary cycle ${c}'$, whose label satisfies $\omega({c}')=r$ in~$F$.}   

The boundary cycle $c$ of $\Pi$ is equivalent to $c'$ by graph transformations, in the above terminology. Namely, there is a reduced labeled graph $\Omega'$ equivalent to $\Omega$ such that the lift of $\Pi$ in  $\widetilde{\Omega'}$ is a copy of $\Pi$ with boundary cycle $c$. In this case, we say that $c$ \emph{lifts to $\widetilde{\Omega'}$ with~$\Pi$}.  We say that a subpath $p$ of $c$ \emph{lifts to  $\widetilde{\Omega}$ with~$\Pi$}, if the sequence of graph transformations from $\Omega'$ to $\Omega$ fixes the edges of $p$.

\begin{remark}
The lift of $\Pi$ in $\widetilde{\Omega}$ is unique. Indeed, given a second lift in $\widetilde{\Omega}$, two distinct reduced labeled cycles in $\Omega$ had a label equal to $r$. This contradicts the graphical small cancellation condition on $\Omega$.
\end{remark}
 
\begin{definition*}[cf. \cite{ollivier_small_2006}*{p. 81}] Let $\Pi_1,\Pi_2$ be faces in $D$ and let $\Omega'$ be a graph equivalent to $\Omega$. An edge $e\in \partial\Pi_1\cap\partial\Pi_2$ \emph{originates} in  $\Omega'$ if  $e$ lifts with $\Pi_1$ and $\Pi_2$ to $\widetilde{\Omega'}$ and these lifts of $e$ coincide. An inner segment  in $D$ \emph{originates}, or is an \emph{originating segment}, if all of its edges originate in such a graph $\Omega'$. 
\end{definition*}

The inner segments originating in $R$ as defined on page~\pageref{D: originate in R} are included in this definition.

\begin{remark}
Every inner segment that does \emph{not} originate  satisfies the $C'_*(\lambda)$--condition, as every such inner segment is a piece in a graph $\Omega'$ equivalent to~$\Omega$.
\end{remark}

 \begin{remark}\label{R: faces in widetilde D} If an originating segment $s$ in the common boundary of two faces $\Pi_1$ and $\Pi_2$ with boundary cycles $sp_1$ and $sp_2$ originates, then \emph{either}  $\omega(p_1)=\omega(p_2)$ in $F$ \emph{or} $\omega(p_1p_2^{-1})$ equals to the label of a reduced cycle in $\Omega$.
 \end{remark}

 
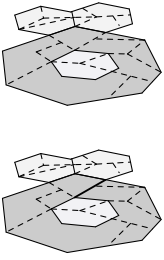
\begin{wrapfigure}{l}{0.30\textwidth}
\definecolor{ttqqqq}{rgb}{0,0,0}
\definecolor{anti-flashwhite}{rgb}{0.95, 0.95, 0.96}
\begin{tikzpicture}[line cap=round,line join=round,>=triangle 45,x=0.6cm,y=0.27cm]
\clip(-1.3,-1) rectangle (2.8,6.1);
\fill[color=ttqqqq,fill=ttqqqq,fill opacity=0.05] (-0.32,4.72) -- (0.34,4.44) -- (0.91,4.87) -- (0.82,5.58) -- (0.16,5.86) -- (-0.41,5.43) -- cycle;
\fill[color=ttqqqq,fill=ttqqqq,fill opacity=0.05] (0.84,5.6) -- (0.91,4.87) -- (1.58,4.57) -- (2.17,5) -- (2.1,5.73) -- (1.43,6.03) -- cycle;
\fill[color=ttqqqq,fill=ttqqqq,fill opacity=0.2] (-0.68,3.66) -- (0.91,4.87) -- (2.44,4.2) -- (2.8,2.62) -- (2.38,1.6) -- (0.74,1) -- (-0.6,2) -- cycle;
\fill[color=ttqqqq,fill=anti-flashwhite,fill opacity=2] (0.62,2.42) -- (1.36,2.26) -- (1.87,2.82) -- (1.64,3.54) -- (0.9,3.7) -- (0.39,3.14) -- cycle;
\draw [color=ttqqqq] (-0.32,4.72)-- (0.34,4.44);
\draw [color=ttqqqq] (0.34,4.44)-- (0.91,4.87);
\draw [color=ttqqqq] (0.82,5.58)-- (0.16,5.86);
\draw [color=ttqqqq] (0.16,5.86)-- (-0.41,5.43);
\draw [color=ttqqqq] (-0.41,5.43)-- (-0.32,4.72);
\draw [color=ttqqqq] (0.84,5.6)-- (0.91,4.87);
\draw [color=ttqqqq] (0.91,4.87)-- (1.58,4.57);
\draw [color=ttqqqq] (1.58,4.57)-- (2.17,5);
\draw [color=ttqqqq] (2.17,5)-- (2.1,5.73);
\draw [color=ttqqqq] (2.1,5.73)-- (1.43,6.03);
\draw [color=ttqqqq] (1.43,6.03)-- (0.84,5.6);
\draw [color=ttqqqq] (-0.68,3.66)-- (0.91,4.87);
\draw [color=ttqqqq] (0.91,4.87)-- (2.44,4.2);
\draw [color=ttqqqq] (2.44,4.2)-- (2.8,2.62);
\draw [color=ttqqqq] (2.8,2.62)-- (2.38,1.6);
\draw [color=ttqqqq] (2.38,1.6)-- (0.74,1);
\draw [color=ttqqqq] (0.74,1)-- (-0.6,2);
\draw [color=ttqqqq] (-0.6,2)-- (-0.68,3.66);
\draw [color=ttqqqq] (0.62,2.42)-- (1.36,2.26);
\draw [color=ttqqqq] (1.36,2.26)-- (1.87,2.82);
\draw [color=ttqqqq] (1.87,2.82)-- (1.64,3.54);
\draw [color=ttqqqq] (1.64,3.54)-- (0.9,3.7);
\draw [color=ttqqqq] (0.9,3.7)-- (0.39,3.14);
\draw [color=ttqqqq] (0.39,3.14)-- (0.62,2.42);
\draw [dash pattern=on 2pt off 2pt] (0.9,3.7)-- (1.58,4.57);
\draw [dash pattern=on 2pt off 2pt] (0.34,4.44)-- (0.9,3.7);
\draw [dash pattern=on 2pt off 2pt] (0.61,4.08)-- (0.06,3.62);
\draw [dash pattern=on 2pt off 2pt] (0.06,3.62)-- (0.39,3.14);
\draw [dash pattern=on 2pt off 2pt] (0.39,3.14)-- (-0.6,2);
\draw [dash pattern=on 2pt off 2pt] (0.39,3.14)-- (1.02,3.06);
\draw [dash pattern=on 2pt off 2pt] (1.02,3.06)-- (1.64,3.54);
\draw [dash pattern=on 2pt off 2pt] (1.02,3.06)-- (2.61,2.16);
\draw [dash pattern=on 2pt off 2pt] (2.13,2.43)-- (1.7,1.48);
\draw [dash pattern=on 2pt off 2pt] (1.64,3.54)-- (2.06,3.52);
\draw [dash pattern=on 2pt off 2pt] (2.06,3.52)-- (2.44,4.2);
\draw [dash pattern=on 2pt off 2pt] (2.06,3.52)-- (2.8,2.62);
\draw [dash pattern=on 2pt off 2pt] (2.44,4.2)-- (1.34,4.26);
\draw [dash pattern=on 2pt off 2pt] (0.91,4.87)-- (1.38,5.34);
\draw [dash pattern=on 2pt off 2pt] (1.38,5.34)-- (2.13,5.44);
\draw [dash pattern=on 2pt off 2pt] (1.74,5.39)-- (1.88,4.92);
\draw [dash pattern=on 2pt off 2pt] (1.38,5.34)-- (1.07,5.77);
\draw [dash pattern=on 2pt off 2pt] (0.16,5.86)-- (0.22,5.28);
\draw [dash pattern=on 2pt off 2pt] (0.22,5.28)-- (-0.32,4.72);
\draw [dash pattern=on 2pt off 2pt] (0.22,5.28)-- (1.38,5.34);
\draw (0.82,5.58)-- (0.84,5.6);
\draw [dash pattern=on 2pt off 2pt] (-0.32,4.72)-- (0.9,5.32);
\end{tikzpicture}
\begin{tikzpicture}[line cap=round,line join=round,>=triangle 45,x=0.6cm,y=0.27cm]
\clip(-1.3,-1) rectangle (2.8,6.1);
\fill[color=ttqqqq,fill=ttqqqq,fill opacity=0.05] (-0.32,4.72) -- (0.34,4.44) -- (0.91,4.87) -- (0.82,5.58) -- (0.16,5.86) -- (-0.41,5.43) -- cycle;
\fill[color=ttqqqq,fill=ttqqqq,fill opacity=0.05] (0.84,5.6) -- (0.91,4.87) -- (1.58,4.57) -- (2.17,5) -- (2.1,5.73) -- (1.43,6.03) -- cycle;
\fill[color=ttqqqq,fill=ttqqqq,fill opacity=0.2] (-0.68,3.66) -- (0.91,4.87) -- (2.44,4.2) -- (2.8,2.62) -- (2.38,1.6) -- (0.74,1) -- (-0.6,2) -- cycle;
\fill[color=ttqqqq,fill=anti-flashwhite,fill opacity=2] (0.62,2.42) -- (1.36,2.26) -- (1.87,2.82) -- (1.64,3.54) -- (0.9,3.7) -- (0.39,3.14) -- cycle;
\draw [color=ttqqqq] (-0.32,4.72)-- (0.34,4.44);
\draw [color=ttqqqq] (0.34,4.44)-- (0.91,4.87);
\draw [color=ttqqqq] (0.82,5.58)-- (0.16,5.86);
\draw [color=ttqqqq] (0.16,5.86)-- (-0.41,5.43);
\draw [color=ttqqqq] (-0.41,5.43)-- (-0.32,4.72);
\draw [color=ttqqqq] (0.84,5.6)-- (0.91,4.87);
\draw [color=ttqqqq] (0.91,4.87)-- (1.58,4.57);
\draw [color=ttqqqq] (1.58,4.57)-- (2.17,5);
\draw [color=ttqqqq] (2.17,5)-- (2.1,5.73);
\draw [color=ttqqqq] (2.1,5.73)-- (1.43,6.03);
\draw [color=ttqqqq] (1.43,6.03)-- (0.84,5.6);
\draw [color=ttqqqq] (-0.68,3.66)-- (0.91,4.87);
\draw [color=ttqqqq] (0.91,4.87)-- (2.44,4.2);
\draw [color=ttqqqq] (2.44,4.2)-- (2.8,2.62);
\draw [color=ttqqqq] (2.8,2.62)-- (2.38,1.6);
\draw [color=ttqqqq] (2.38,1.6)-- (0.74,1);
\draw [color=ttqqqq] (0.74,1)-- (-0.6,2);
\draw [color=ttqqqq] (-0.6,2)-- (-0.68,3.66);
\draw [color=ttqqqq] (0.62,2.42)-- (1.36,2.26);
\draw [color=ttqqqq] (1.36,2.26)-- (1.87,2.82);
\draw [color=ttqqqq] (1.87,2.82)-- (1.64,3.54);
\draw [color=ttqqqq] (1.64,3.54)-- (0.9,3.7);
\draw [color=ttqqqq] (0.9,3.7)-- (0.39,3.14);
\draw [color=ttqqqq] (0.39,3.14)-- (0.62,2.42);
\draw [line width=0.7pt,color=ttqqqq] (0.9,3.7)-- (1.58,4.57);
\draw [dash pattern=on 2pt off 2pt] (0.34,4.44)-- (0.9,3.7);
\draw [dash pattern=on 2pt off 2pt] (0.61,4.08)-- (0.06,3.62);
\draw [dash pattern=on 2pt off 2pt] (0.06,3.62)-- (0.39,3.14);
\draw [dash pattern=on 2pt off 2pt] (0.39,3.14)-- (-0.6,2);
\draw [dash pattern=on 2pt off 2pt] (0.39,3.14)-- (1.02,3.06);
\draw [dash pattern=on 2pt off 2pt] (1.02,3.06)-- (1.64,3.54);
\draw [dash pattern=on 2pt off 2pt] (1.02,3.06)-- (2.61,2.16);
\draw [dash pattern=on 2pt off 2pt] (2.13,2.43)-- (1.7,1.48);
\draw [dash pattern=on 2pt off 2pt] (1.64,3.54)-- (2.06,3.52);
\draw [dash pattern=on 2pt off 2pt] (2.06,3.52)-- (2.44,4.2);
\draw [dash pattern=on 2pt off 2pt] (2.06,3.52)-- (2.8,2.62);
\draw [dash pattern=on 2pt off 2pt] (2.44,4.2)-- (1.34,4.26);
\draw [dash pattern=on 2pt off 2pt] (0.91,4.87)-- (1.38,5.34);
\draw [dash pattern=on 2pt off 2pt] (1.38,5.34)-- (2.13,5.44);
\draw [dash pattern=on 2pt off 2pt] (1.74,5.39)-- (1.88,4.92);
\draw [dash pattern=on 2pt off 2pt] (1.38,5.34)-- (1.07,5.77);
\draw [dash pattern=on 2pt off 2pt] (0.16,5.86)-- (0.22,5.28);
\draw [dash pattern=on 2pt off 2pt] (0.22,5.28)-- (-0.32,4.72);
\draw [dash pattern=on 2pt off 2pt] (0.22,5.28)-- (1.38,5.34);
\draw (0.82,5.58)-- (0.84,5.6);
\draw [dash pattern=on 2pt off 2pt] (-0.32,4.72)-- (0.9,5.32);
\end{tikzpicture} 
\caption{Diagrams $D$, dashed edges originate.}
\label{F: diagrams D}
\end{wrapfigure}

 Note that the label $\omega(p_1p_2^{-1})$ is not necessarily in the set $R$. Therefore, a minimality argument as in the proof of Corollary \ref{C: Ollivier} cannot be directly applied. In particular, Theorem \ref{T: classical small cancellation lemma} cannot be applied. We provide an extension of Corollary \ref{C: Ollivier} to the graphical setting: Corollary~\ref{C: Ollivier3} below. 

\smallskip

A \emph{region} is a connected subdiagram of $D$ which is defined inductively as follows. Every face of $D$ is a region. If $M$ is a region and $\Pi$ is a face of $D$, then $M\cup \Pi$ is a region if $M$ and $\Pi$ have an originating inner segment of $D$ in common.  A region is \emph{maximal} if it is not contained in a region with a larger number of faces. 
 Let $M$ be a maximal region.  
 In Figure \ref{F: diagrams D} the dark-grey diagrams $M$ are not simply-connected. \\
\begin{lemma}\label{Hilfssatz van Kampen} Let $D$ be a minimal van Kampen diagram over $R$ as above, let $M$ be a maximal region.
\begin{enumerate}
  \item The label of every boundary cycle of $M$ equals in $F$ to the label of a reduced cycle in $\Omega$.
\item If $M$ is simply-connected in $D$, then we can assume that $M$ has a reduced boundary cycle. 
\end{enumerate}
\end{lemma}

\begin{proof}
The  minimality argument of the proof of Corollary \ref{C: Ollivier} implies that every boundary cycle of $M$ represents a non-trivial element of $F$. 
 Moreover, for all  connected subregions $M'\subseteq M$  this argument implies that the label of a boundary cycle of $ M'$ equals a non-trivial element in $F$. By induction, with 
 Remark \ref{R: faces in widetilde D} as the base case, we conclude that the label of every boundary cycle of $M$ equals in $F$ to the label of a non-trivial cycle in $\Omega$. This implies assertion~1.

 \medskip
 
Let $M$ be simply-connected in $D$ and its boundary cycle, denoted by $c$, be not reduced. Assertion 1 implies that $c$ is equivalent to a reduced cycle $c'$  whose label equals the label of $c$ in $F$. The cycle $c'$ can be obtained by applying Reductions to $c$ as long as possible. We then apply Deletions. The procedure terminates as $c$ represents a non-trivial element of $F$ by assertion 1.
 
 \medskip
 
  { We now prove assertion 2.} We remove the interior of $M$ from $D$. We obtain a diagram~$D'$. Either $\partial M$ contains an exterior edge of $D$ or not. In the second case, the boundary of $D'$ consists of $\partial D$ and $c$. Apply the sequence of Reductions that transform $c$ to $c'$ to the copies of $c$ in the boundary of $M$ and ${D}'$. We denote the resulting diagrams by $M'$ and ${D}''$. Then apply Deletions to both $M'$ and ${D}''$. The boundary cycle  of $M'$ is $c'$ and the boundary of ${D}''$ consists of  $\partial D$ and a copy of $c'$. By construction, the inner segments of $M$ originate, $|M|=|\Pi'(M)|$, and $|{D}'|=|{D}''|$. The diagram $M'$ can be glued into  $D''$ along the copies of $c'$. We obtain a new van Kampen diagram $D'''$ with boundary word $w$ such that $|D|=|D'''|$.  
In the first case the arguments are analogous.
\end{proof}
This  amounts to the aforementioned generalization of Corollary \ref{C: Ollivier}. 
\begin{cor} \label{C: Ollivier3} Let $D$ be a minimal van Kampen diagram over $R$ as above. Then every maximal region $M$ is simply-connected (hence, the interior of $M$ is homeomorphic to an open disk). 
\end{cor}
\begin{proof}
 We use the notions of the proof of Corollary \ref{C: Ollivier} and consider an innermost not simply-connected region $M$ in $D$. All maximal regions $M_i$ of $K_2$ are simply-connected by assumption. This allows to define a diagram  $\widetilde{K_2}$ with faces $ M_i$ by deleting the originating edges from $K_2$. By definition, the faces $M_i$ of $\widetilde{K_2}$ are simply-connected. The $ M_i$ have reduced labeled boundary cycles by Lemma~\ref{Hilfssatz van Kampen}. Every face $M_i$ has a lift to a graph $\Omega$ and every arc $p$ in $\widetilde{K_2}$ lifts with a face $ M_i$ to such $\Omega$. Moreover, there is a reduced cycle in $\Omega$ whose label equals the label of $K_2$ in $D$. As $p$ does not consist of originating edges, we conclude that $p$ is a piece. This implies that $\widetilde{K_2}$ satisfies the $C'_*(1/6)$--condition and all arcs on the boundary of $\widetilde{K_2}$ are pieces. This contradicts Theorem \ref{T: classical small cancellation lemma}.
 \end{proof}

Corollary \ref{C: Ollivier3} allows to define a diagram $\widetilde{D}$ with faces $M$, by deleting the originating edges from $D$. In addition, all faces of $\widetilde{D}$  are simply-connected. As no inner segment of  $\widetilde{D}$ originates, $\widetilde{D}$ satisfies the $C'_*(\lambda)$--condition and we can apply Theorem~\ref{T: classical small cancellation lemma} to $\widetilde{D}$. The result summarizes as follows.

\begin{lemma}[Graphical small cancellation lemma] \label{L: graphical small cancellation lemma}
Let $0< \lambda \leqslant 1/6$. Let $D$ be a minimal van Kampen diagram over a $Gr_*'(\lambda)$--graphical small cancellation presentation. 
 \begin{itemize}
    \item If $\widetilde{D}$ has more than two faces then there are at least two exterior faces $M$ in $\widetilde{D}$ such that  
      $$\lvert\omega(\partial_{ext}M)\rvert_*>\left( 1-{3}{\lambda}\right){\lvert\omega(\partial M)\rvert_*},$$ 
      $\partial_{int}M$ consists of at most three pieces, and $\partial_{ext}M$ is connected.
    \item The following inequality is satisfied:
    $$\lvert\omega(\partial D)\rvert_*> (1-6\lambda) \sum_{ M_i \text{ is a face in }\widetilde{D}} \lvert\omega(\partial M_i)\rvert_*.$$  
    \item The label of $\partial D$ is at  least as long as (with respect to the length function $\lvert\, .\, \rvert_*$) the label of $\partial M$, for every face $M$ in $\widetilde{D}$.
 \end{itemize}
\end{lemma}

 We  apply the techniques of this section to show our graphical small cancellation theorem.

\subsection{Small cancellation theorems}

We now generalize the classical small cancellation theorems, Theorem~\ref{C: hyperbolic},  Theorem~\ref{C: t.f.} and Theorem~\ref{C: injection} to the graphical setting. 
Our first aim is to show a linear isoperimetric inequality for all minimal van Kampen diagrams over $R$ given by a labeled graph $\Omega$ with the $Gr_*'(1/6)$--condition. 

\smallskip

Let $D$ be a minimal van Kampen diagram over $R$ and $\widetilde{D}$ is obtained from $D$ as above, by deleting originating edges.  Lemma \ref{L: graphical small cancellation lemma} yields: \[\lvert \omega(\partial D)\rvert_*> (1-6\lambda) \sum_{ M_i \text{ is a face in }\widetilde{D}} \lvert\omega(\partial M_i) \rvert_*.\]  

Every face $M$ of $\widetilde{D}$ represents a connected and simply-connected region (also denoted by $M$) of $D$  all of whose inner edges originate. The label on a boundary cycle $c$ of $M$ is weakly reduced, represents an element of $F$, and lifts with $M$ to  $\widetilde{\Omega}$ (or to $\widetilde{\Omega}'$, for a graph $\Omega'$ equivalent to $\Omega$). As all inner segments of $M$ originate, the whole  diagram $M$ lifts to $\widetilde{\Omega}$. Thus, there is an equivalent diagram $M'$ that has an immersion into $\widetilde{\Omega}$: its boundary cycle is a copy of $c$ and its faces are the lifts of the faces of $M$ in $\widetilde{\Omega}$. We have $|M|=|M'|$.  
 A combination of the above inequality with the next lemma yields a linear isoperimetric inequality for $D$.

\begin{lemma}
Let $D'$ be a van Kampen diagram that is immersed in $\widetilde{\Omega}$. Then there is $C>0$ such that $$\lvert D'\rvert\leqslant C\lvert \omega(\partial D') \rvert_*.$$ 
\end{lemma}
\begin{proof} We extend the proof of \cite{ollivier_small_2006}*{p.81, L.11} for the word length metric to  free product length~$\lvert\, .\, \rvert_*$. It suffices to prove the claim for a special choice of $ R $. Denote by $\diam$ the diameter of $\Omega$. Let $ R $ be the elements represented by labels of cycles $c$ such that $\lvert\omega(c) \rvert_*\leqslant3 \diam$. Let  $D'$ be a minimal van Kampen diagram for $w$ over $ R $ that is immersed in  $\widetilde{\Omega}$. We prove that $\lvert D\rvert\leqslant\frac{3\lvert w\rvert_*}{\gamma}$, where $\gamma=\min\{\lvert \omega(c) \rvert_* \mid c \text{ is a non-trivial cycle in } \Omega\}$.

If $\lvert w\rvert_*\leqslant2 \diam$, there is a diagram with one face. If $$2\diam\leqslant\lvert w\rvert_* \leqslant(n+1)\diam,$$ $w$ is the concatenation of paths $w'$ and $w''$,  so that $\lvert w'\rvert_*=2\diam$. There is a path $x$ connecting the terminal and the starting vertex of $w'$ such that $$\lvert x \rvert_*\leqslant\lvert x\rvert\leqslant\diam;$$ and w.l.o.g. (choose $w'$ such that $\lvert w''\rvert_*=\lvert w\rvert_*-\lvert w'\rvert_*$) $$\lvert xw''\rvert\leqslant\lvert w\rvert_*-\diam \leqslant n\diam.$$

Observe that $\lvert w'x\rvert_*\leqslant3\diam$. There is a van Kampen diagram for $w'x$ consisting of only one face. Iterating yields that $D'$ has at most $1+\frac{\lvert w\rvert_*}{\diam}$ faces. Since $\diam \geqslant \frac{\girth}{2}\geqslant \frac{\gamma}{2}$ and $\lvert w\rvert_*\geqslant \gamma$, we have $\lvert D'\rvert\leqslant\frac{3\lvert w\rvert_*}{\gamma}.$
\end{proof}

We now use the linear isoperimetric inequality for $D$ to conclude our first main result. 

\begin{theorem}\label{lii}
Let $G_1,\ldots,G_d$ be finitely generated groups.  Let $\Omega$ be a family of finite connected graphs with a reduced labeling by $G_1\cup \ldots \cup G_d$. Suppose the labeling satisfies the $Gr_*'(1/6)$--condition over the free product $G_1*\cdots *G_d$.  Let $G$ be the group generated by $G_1\cup \ldots \cup G_d$ subject to relators represented by labels of simple cycles generating the fundamental group of $\Omega$. Let $D$ be a minimal van Kampen diagram over the given presentation of $G$.

Then $D$ satisfies the linear isoperimetric inequality $$\lvert D\rvert\leqslant C \lvert \partial D\rvert_*.$$

If $\Omega$ is finite and $G_1,\ldots,G_d$ are Gromov hyperbolic, then $G$ is Gromov hyperbolic.
\end{theorem}

This extends Theorem \ref{C: hyperbolic} to the free product. The next result extends Theorem \ref{C: t.f.}.

\begin{theorem}\label{T: tf} Let $G_1,\ldots,G_d$ be finitely generated torsion-free groups.  Let $\Omega$ be a family of finite connected graphs with a reduced labeling by $G_1\cup \ldots \cup G_d$. Suppose the labeling satisfies the $Gr_*'(1/8)$--condition over the free product $G_1*\cdots *G_d$. Let $G$ be the group presented by $G_1\cup \ldots \cup G_d$ as generators and the elements $r\in R$ represented by labels of the cycles of $\Omega$ as relators. 

Then $G$ is torsion-free.
\end{theorem}
\begin{proof} We extend the proof of \cite{lyndon_combinatorial_1977}*{p. 281f, Th.10.1, Ch.V}, which does not apply as their Lemma \cite{lyndon_combinatorial_1977}*{p. 281, L.10.2, Ch.V} does not hold in the graphical setting.

Let $\Omega$ and $R$ as above. The $Gr_*'(1/8)$--condition implies that there are no proper powers among relators in $R$. 
Let $w\in F$, $\lvert w\rvert_*>1$.
Let $\lvert z\rvert_*>1$ be an element of least length among all conjugates of $w$ in $G$ of order $n\geqslant 2$ in $G$. (All conjugates have the same order.) Let $D$ be a minimal van Kampen diagram for $z^n$. 

We can assume that the boundary cycles of faces of the corresponding diagram $\widetilde{D}$ are neither proper powers nor concatenations of cycles in $\Omega$. This replaces  Lemma  \cite{lyndon_combinatorial_1977}*{p. 281, L.10.2, Ch.V}. 

Indeed, let $M$ be a face in $\widetilde{D}$, and $\tilde{r}$ the label of $\partial M$. Assume $\tilde{r}=a^m$ is a weakly reduced product with $m\geqslant 2$. As all inner edges of $M$ are originating, $\widetilde{r}$ equals the label of a cycle $c$ in $\Omega$. If $a$ is not the label of a simple cycle $c_a$ in $\Omega$ and $c$ is the concatenation of $m$ copies of $c_a$, then $a$ and $a^{m-1}$ are pieces. This contradicts the $Gr_*'(1/6)$-condition. Thus, $c_a$ is a simple cycle in $\Omega$. Then we replace the face $M$ by the $m$-rose consisting of $m$-faces each of whose boundary cycles is a copy of $c_a$ that are identified at one  common vertex. If $r$ is a concatenation of $m$ cycles $c_i$ with different labels, we replace $M$ by the $m$-rose consisting of $m$-faces with boundaries $c_1,\ldots,c_m$ that are identified at one common vertex. The inner segments in the resulting diagram satisfy the $C'_*(1/8)$--condition and all its faces are simply-connected.

By Lemma  \ref{L: graphical small cancellation lemma}, $z^n=uz'$, where $u$ is a subword of $r\in R$ such that $\lvert u\rvert_*>\frac58\lvert r\rvert_*$. By the above assumption, $r$ equals the label of a single simple cycle $c$ of $\Omega$. {By the minimal length condition} on $z$, $u$ is not a subword of $z$. Hence, we have that $u=z^mt$ is a weakly reduced product, where $m\geqslant 1$ and $t$ does not begin with a power of $z$. Write $z=ts$. 
 Then $r=uv=(ts)^m tv$. 
 
 If $m>1$, as $c$ is not a concatenation of cycles, $(ts)^{m-1}$, $(ts)$, and $t$ are labels of pieces in a reduced labeled graph $\Omega'$ equivalent to $\Omega$. Then $u$ would be the product of the labels of three pieces. Thus, $\lvert u\rvert_*< \frac12 \lvert r\rvert_*$,  a contradiction.
 
 Hence, $r=tstv$. If $ts=tv$, $r$ and $z$ are powers of a common subword. As observed above, this word would be the label of a piece as $c$ is not a concatenation of smaller cycles. This contradicts the $Gr_*'(1/8)$--condition. Thus, $t$ is the label of a piece, and $\lvert t\rvert_*<1/8\min\{\lvert r\rvert_*\mid r\in R\}$. As $u=v$ in $G$, we have $\lvert tv\rvert_*\leqslant \frac48 \min\{\lvert r\rvert_*\mid r\in R\}$. As $\lvert u\rvert_*\geqslant \frac58\min\{\lvert r\rvert_*\mid r\in R\}$, this contradicts the minimality in the choice of $z$.
\end{proof}

Finally, we obtain the following generalization of Theorem \ref{C: injection}.

\begin{theorem}\label{T: gi} Let $G_1,\ldots,G_d$ be finitely generated groups.  Let $\Omega=(\Omega_l)_l$ be a family of finite connected graphs with a reduced labeling by $G_1\cup \ldots \cup G_d$. Suppose $\Omega$ satisfies $Gr_*'(1/6)$--condition over the free product $G_1*\cdots *G_d$. Let $G$ be the group presented by $G_1\cup \ldots \cup G_d$ as generators and the elements $r\in R$ represented  by labels of the cycles of $\Omega$ as relators. 

Then every graph $\Omega_l$ of $\Omega$ injects in the Cayley graph of $G(\Omega)$ with respect to $G_1\cup \ldots \cup G_d$. 

If $\Omega$ is labeled by $X\sqcup X^{-1}$, then each graph $\Omega_l$ of $\Omega$ injects into the Cayley graph of $G(\Omega)$ with respect to $X\sqcup X^{-1}$.
\end{theorem}

\begin{proof} Let $\Gamma:=\Omega_l$ for some $l$.  Choose two distinct vertices $v_1$ and $v_2$ of $\Gamma$. Let $p$ be a non-trivial simple path in $\Gamma$ starting at $v_1$ and terminating at $v_2$. Let $x$ be the label of $p$. 
 As the labeling is reduced, $x$ represents a non-trivial element of $F$. We argue by contradiction and assume that $x=1$ in $G(\Gamma)$; equivalently, the images of $v_1$ and $v_2$ are identified in the Cayley graph of $G(\Gamma)$.
 
 Let $D$ be a minimal van Kampen diagram for $x$ over $R$. We assume that $D$ has no inner edges originating in $\Gamma$. Otherwise, we put $D:=\widetilde{D}$. The diagram $D$ has a distinguished vertex $v$. This is the starting and terminal vertex of its boundary cycle $c$, whose label is $x$.  The vertex $v$ has two \emph{lifts} to $\Gamma$, the vertices $v_1$ and $v_2$. We say that $c$ lifts to $p$ in $\Gamma$. 
Observe that $c$ can be obtained from $p$ by applying the graph transformations: the AO-moves, Reductions, and Deletions. Let us apply the same sequence of AO-moves and Reductions to the path $p$ in $\Gamma$. We then apply Deletions until no degree-one vertex is left. Let $\Gamma'$ be the resulting graph. It is equivalent to $\Gamma$. The image of $p$ in $\Gamma'$ is a copy of $c$ with starting vertex $v_1$ and terminal vertex $v_2$.
 
 Therefore, we assume that $p$ is a copy of $c$ in $\Gamma$.
 
 We choose $p$ so that the number of faces of $D$ is minimal among all minimal van Kampen diagrams for elements in $F$ represented by the labels of paths between $v_1$ and $v_2$ in $\Gamma$. (Such labels are not equal in~$F$.) If $\Pi$ is an exterior face in $D$ let  $s$ denote a path in $\partial_{ext}\Pi$. If the lift of $\Pi$ to $\widetilde{\Gamma}$ is such that the lift of $s$ with $\Pi$ coincides with the lift of $c$ to $p$ , then $\Pi$ is called \emph{originating with $c$}. 
 
 In $D$, no face originates with $c$. If there was such a face $\Pi$, we could remove it from $D$. Specifically, we remove the interior of the face as well as the path $s$. If this deletion disconnects $D$, then we further delete all faces that are not in the connected component containing $v$. The boundary cycle $\bar{c}$ of the resulting diagram, starting and terminating at $v$, has a lift to  a path $\bar{p}$ in $\Gamma$, and the labels of $p$ and $\bar{p}$  are not equal in~$F$. (More specifically, there is a graph $\Gamma'$ so that $\bar{c}$ is a copy of a path $\bar{p}'$ in $\Gamma'$ starting at $v_1$ and terminating at $v_2$, and $\bar{p}$ is the preimage in $\Gamma$ of $\bar{p}'$.)  This diagram for $\omega(\bar{p})$  has a lesser number of faces, a contradiction to the minimality condition in the choice of $p$.
 
 The path $p$ lies on a  non-trivial cycle $d=pq$ in $\Gamma$. Let $z$ be the label of $d$  and $y$ the label of $q$ so that $z=xy$. There is no cancellation  between $x$ and $y$.  
  Let $D'$ be a copy of the 2-cell of $\widetilde{\Gamma}$ with boundary cycle $d$. Recall that $p$ is a copy of $c$. Let us glue $D$ and $D'$ along the paths $c$ and $p$. Let $D''$ be the resulting diagram. The boundary label of $D''$ equals $y$. The image of $D'$ in $D''$ is, as a face in $D''$, not simply-connected as one vertex of its boundary is identified.
 
 As there are no faces in  $D$ originating with $c$, the inner segments $s$ of $D''$ in $c$ do not originate. Otherwise, there is a face $\Pi$ of $D$ such that the lifts of $s$ with $\Pi$ and $D'$ are equal. In particular, $\Pi$ originates with $c$ along $s$, a contradiction. Therefore, $D''$ satisfies the $C'_*(1/6)$--condition. Lemma \ref{L: graphical small cancellation lemma} implies that such a diagram cannot exist. This finishes the proof. 
\end{proof}

This theorem gives a new instance of a theorem of Gromov \cite{gromov_random_2003}*{p. 141, S.4.8, see also p.75, S.1.1}. Our result generalizes Ollivier's variant of Gromov's theorem \cite{ollivier_small_2006}*{p. 77, Th.1}. 
 Another generalization, under the combinatorial $C(6)$--small cancellation condition, is given in~\cite{gruber_graphical_2012}.

We apply the above theorems to finite and infinite families of labeled reduced graphs that then define torsion-free groups without the unique product property.


\section{Rips-Segev groups revisited}

\label{rssmallcanc}

Rips-Segev  define a group by a graph with directed edges. Every edge of their graph is labeled by a word in $a^{\pm1},b^{\pm1}$. The letters $a$ and $b$ freely generate a free group { $F:=\langle a \rangle *\langle b\rangle$.} Let $R$ be the infinite set of elements in $F$ that are represented by the labels of all non-trivial cycles in their graph. Rips-Segev's group is given by a presentation with generators $a,b$ and relators $R$. These relators are chosen to imply the non-unique product property. To conclude that the groups are torsion-free and indeed without the unique product property, Rips-Segev first affirm that $R$ satisfies certain classical generalizations of small cancellation conditions. Then they refer to classical results such as the Greendlinger lemma \cite{lyndon_combinatorial_1977}*{p. 278, Th. 9.3, Ch. V} and the torsion theorem \cite{lyndon_combinatorial_1977}*{p. 281 , Th. 10.1, Ch. V}.

\smallskip

The purpose of this section is to explain Rips-Segev's line of argument. This reveals that their reasoning does not allow to conclude that their groups are torsion-free and without the unique product property. In particular, the above-cited results are invalid for the small cancellation conditions used by Rips-Segev. We fill this gap using our graphical small cancellation theory over the free product.  Moreover, we explain that the small cancellation conditions Rips-Segev appeal to, do not a priori imply that their groups are Gromov hyperbolic.

\smallskip

We first review the classical generalizations of small cancellation theory, see e.g. \cite{lyndon_combinatorial_1977}*{p. 267ff, Ch.V.8}.

\subsection{Classical generalizations of small cancellation theory}

An \emph{$R$-sequence} for a word $w$ is a sequence  of relators $r_1,\ldots,r_n \in R$ such that $w=_F\prod_{i=1}^n u_ir_iu_i^{-1}$, where $u_i \in F$ and the equality is in $F$. An  $R$-sequence $r_1,\ldots,r_n$ for $w$ is called  \emph{minimal} if $n$ is minimal among the $R$-sequences for $w$. 

A van Kampen diagram for a word $w$ over $R$ in $F$ is \emph{minimal} if the number of its faces is minimal among the van Kampen diagrams for $w$. A minimal $R$-sequence $r_1,\ldots,r_n$ for $w$ corresponds to a minimal van Kampen diagram the boundary cycles of whose faces are labeled by $r_1,\ldots,r_n$. Lyndon and Schupp studied minimal van Kampen diagrams to solve the word and conjugacy problem for the classical $C'(1/6)$--small cancellation groups \cites{lyndon_dehn_1966,schupp_dehn_1968}. In their terminology, minimal van Kampen diagrams are \emph{diagrams of minimal $R$-sequences}.

\smallskip

Appel-Schupp subsequently developed the \emph{$C(4)$-$T(4)$--small cancellation condition for minimal sequences} \cite{appel_conjugacy_1972}*{p.333}. 
Rips-Segev's first reference  to the small cancellation theory \cite{rips_torsion-free_1987}*{p. 123} points to these conditions. Appel-Schupp's conditions imply the small cancellation conditions only on minimal van Kampen diagrams (in contrast to all reduced van Kampen diagrams as in the case of the classical $C'(\lambda)$--condition).

Suppose that the set of relators $R$ satisfies the following conditions.
\begin{enumerate}
\item[(1)] The relators in $R$ have length 4.
 \item[(2)] If $r_1,r_2 \in R$ cancel two or more letters, then either $r_2=r_1^{-1}$ or $r_1r_2$ is in $R$. 
 \item[(3)] If $r_1,r_2,r_3 \in R$, and there is cancellation in all the products $r_1r_2$, $r_2r_3$, and $r_3r_1$, then $r_1r_2r_3$ is a product of at most two elements of $R$. 
\end{enumerate}
If conditions $(1)$--$(3)$ hold, the corresponding group presentation is said to satisfy the \emph{$C(4)$-$T(4)$--small cancellation condition for minimal sequences (or minimal van Kampen diagrams)}.

Conditions (1) and (2) unify to the \emph{$C(4)$--condition for minimal sequences (or minimal van Kampen diagrams)} over $R$, which states that every inner face of a minimal van Kampen diagram   has 4 inner segments. Condition (3) is called the \emph{$T(4)$--condition}. Note that, as relators have length 4, the set $R$ is finite whenever the generating set $X$ is finite. 

\smallskip

The given geometry of minimal van Kampen diagrams under the $C(4)$-$T(4)$--condition allows to solve the word and the conjugacy problems for groups given by presentations satisfying the $C(4)$-$T(4)$--condition for minimal van Kampen diagrams  \cite{lyndon_combinatorial_1977}*{p.271, L.8.4, Ch.V}.

\smallskip

Gromov's hyperbolicity of a finitely presented group is characterized by a linear word problem: given a finite presentation of the group the minimal van Kampen diagrams over this presentation have to satisfy a linear isoperimetric inequality (w.r.t. the word length metric).  
In contrast, a group with the $C(4)$-$T(4)$--condition for minimal sequences can have a quadratic word problem. Thus, the $C(4)$-$T(4)$--condition is not sufficient to imply that the group is Gromov hyperbolic. 

\smallskip

We have seen the classical generalizations of the small cancellation theory as discussed in \cite{lyndon_combinatorial_1977}*{p. 267, Ch.V.8}. As already mentioned, Rips-Segev do refer to 
Appel-Schupp's conditions for minimal sequences. At the same time, they implicitly further generalize these small cancellation conditions for minimal sequences. However,
this does not complete the arguments as there are no such conditions explained in their reference \cite{lyndon_combinatorial_1977}. We now explain these small cancellation conditions.

\subsection{Rips-Segev's arguments}

In order to conclude that their groups are torsion-free and without the unique product property, Rips-Segev suggest two independent lines of argument \cite{rips_torsion-free_1987}*{p. 123}: They refer to a ``$C([p/2])$--condition for minimal sequences'' in \cite{lyndon_combinatorial_1977} ($p$ is a large number depending on the graph which defines their group presentation). Alternatively, they refer to a  ``$C'_*(1/[p/2])$--condition for minimal sequences over the free product $\langle a\rangle *\langle b\rangle$'' in \cite{lyndon_combinatorial_1977}. Both conditions are not described in \cite{lyndon_combinatorial_1977}*{Ch. V} and they are not defined in Rips-Segev's paper. 

Rips-Segev wish to apply the results of the classical small cancellation theory \cite{lyndon_combinatorial_1977}. Rips-Segev refer to the Greendlinger lemma \cite{lyndon_combinatorial_1977}*{p. 278, Th. 9.3, Ch. V} and the torsion theorem \cite{lyndon_combinatorial_1977}*{p. 281 , Th. 10.1, Ch. V}. The Greendlinger lemma is indispensable in their argument to imply the non-unique product property, the torsion theorem is required to imply torsion-freeness.
  
  \smallskip
  
We now propose definitions for both Rips-Segev's alternatives. Then we analyze the (non)-availability of the Greendlinger lemma and the torsion theorem.
Our definitions are based on  Rips-Segev's comment on \cite{rips_torsion-free_1987}*{p. 123} and the basic definitions from \cite{lyndon_combinatorial_1977}: the $C(4)$-$T(4)$--condition for minimal sequences and the $C'_*(\lambda)$--condition. 
 
 \smallskip
 
We say that $R$ satisfies the \emph{$C([p/2])$--condition for minimal sequences} if, in every minimal van Kampen diagram over $R$, each inner face has at least $[p/2]$ inner segments. The Greendlinger lemma is not available under the $C([p/2])$--condition for minimal sequences. The torsion theorem \cite{lyndon_combinatorial_1977}*{p. 281 , Th. 10.1, Ch. V} does not apply if we replace the $C'_*(1/8)$--condition by the $C([p/2])$--condition for minimal sequences (or even by the aforementioned $C(4)$-$T(4)$--small cancellation for minimal sequences).  Thus, Rips-Segev's first reference cannot be used to conclude.

\smallskip 
  
We say that $R$ satisfies the \emph{$C'_*(1/[p/2])$--condition for minimal sequences over the free product} if every minimal van Kampen diagram over $R$ satisfies the $C'_*(1/[p/2])$--condition (instead of \emph{all} reduced van Kampen diagrams as  discussed in Section~\ref{S: classical small cancellation theory}). The Greendlinger lemma, for instance, our Theorem~\ref{T: classical small cancellation lemma}, is then applicable. However, the torsion theorem is not available: the proof in \cite{lyndon_combinatorial_1977}*{p.281, L.10.2, Ch.V.9}  is using pieces; and pieces are not available in the small cancellation theory for minimal sequences. Hence, this attempt cannot be used to conclude. 
 
\subsection{Another small cancellation condition}

In view of the above explanations, the small cancellation theory for minimal sequences suggested by Rips-Segev has to be replaced. We propose a new small cancellation condition that has not been investigated in the literature. The idea for its  definition is based on the following generalization of the classical notion of a piece that we extract from \cite{rips_torsion-free_1987}*{p. 117}.

\begin{definition*}\label{D: RSpiece}  
A  \emph{Rips-Segev piece} is an element $p\in F$ such that, for distinct $r_1,r_2\in R$,
 $r_1=up$ and $r_2=p^{-1}v$ are weakly reduced products and $uv$ cancels neither to 1 nor to another relator in $R$. 
\end{definition*}

\begin{definition*}\label{D: RSsmallcancellation}
Let $0<\lambda<1.$
A set of relators $R$ satisfies the \emph{Rips-Segev-$C'_{*}(\lambda )$--small cancellation condition over $F$}, if for every Rips-Segev piece $p$ we have that 
\[
|p|_*< \lambda \min\{ |r|_*\mid r\in R\} .
\]
In this case, the group $G$ is a \emph{Rips-Segev-$C'_{*}(\lambda)$--small cancellation group} over $F$. 
\end{definition*}

Rips-Segev had considered an infinite set $R$. The set $R$ has no proper powers and satisfies the Rips-Segev-$C'_*(1/[p/2])$--condition over $\langle a\rangle * \langle b \rangle$. As explained below, minimal van Kampen diagrams over~$R$ satisfy the $C'_*(1/[p/2])$--condition. Hence, $R$ satisfies the above $C'_*(1/[p/2])$-- and $C([p/2])$--small cancellation conditions for minimal sequences.

\smallskip

Altogether, we feel that the Rips-Segev-$C'_*(\lambda)$--condition is implicitly used in \cite{rips_torsion-free_1987}.

\smallskip

  Again, arguments that the classical results \cite{lyndon_combinatorial_1977}*{p. 278, Th. 9.3, p. 281, Th. 10.1, Ch. V} extend to the Rips-Segev small cancellation are missing. We first discuss the Greendlinger lemma, then we comment on the torsion theorem.
\begin{itemize}

\item Let $0<\lambda \leqslant 1/6$. As we have observed in Section \ref{S: classical small cancellation theory}, if a minimal van Kampen diagram $D$ is reduced, the Greendlinger lemma \cite{lyndon_combinatorial_1977}*{p.278, Th. 9.3, Ch. V} is available. By the proof of \cite{lyndon_combinatorial_1977}*{p.277, L. 9.2(2), Ch. V}, $D$ is reduced if if satisfies the $C'_*(\lambda)$--condition, in the sense of Section~\ref{S: classical small cancellation theory}.  

 We can apply a minimality argument as in the proof of Corollary \ref{C: Ollivier} to conclude. Indeed, if a segment $s$ in the common boundary of two faces $\Pi_1$ and $\Pi_2$ with boundary cycles $sp_1$ and $sp_2$  does \emph{not} satisfy the $C'_*(\lambda)$--condition in the sense of Section~\ref{S: classical small cancellation theory}, then \emph{either}  $\omega(p_1)=\omega(p_2)$ \emph{or} $\omega(p_1p_2^{-1})$ represents a relator in $R$. We conclude that $w$ can be expressed as a product of conjugates of less than $n$ relators,  contradicting the minimality assumption.

This implies the Greendlinger lemma for Rips-Segev-$C'_{*}(\lambda)$--groups.
\item
The torsion theorem cannot be obtained using Lyndon-Schupp's arguments. 
In particular, Lyndon-Schupp's proof uses that if $r=x^ma$ (weakly reduced), $m>1$ and $r$ is not a proper power, then $x$ and $x^{m-1}$ are pieces \cite{lyndon_combinatorial_1977}*{p.281, L. 10.2, Ch. V}. This is false in Rips-Segev's situation. 
\end{itemize} 

Moreover, a Rips-Segev-$C'_{*}(1/6)$--group is a priori not Gromov hyperbolic. As explained above, minimal van Kampen diagrams  over an \emph{infinite} set $R$ satisfy the $C'_{*}(1/6)$--condition. 
However, minimal van Kampen diagrams over a \emph{finite} subset of $R$ do not satisfy the $C'_*(1/6)$--condition as not all inner segments are controlled by the Rips-Segev-$C'_{*}(1/6)$--condition. 
Hence, under the Rips-Segev-$C'_{*}(1/6)$--condition, we cannot conclude that Rips-Segev's group is Gromov hyperbolic. 

\smallskip

In this paper, we close all gaps in Rips-Segev's arguments using our $Gr_*'(1/8)$--condition.  
In the next section, we explain a new construction using our so-called generalized Rips-Segev graphs.  We then derive conditions on the labellings of such graphs that imply the $Gr_*'(1/8)$--condition. We discover that certain reductions of original Rips-Segev graphs are contained in our family of graphs and their labellings satisfy the  $Gr_*'(1/8)$--condition. Finally, we use our general graphical small cancellation theorems to conclude.

\begin{remark} The Rips-Segev-$C'_*(\lambda)$--condition is \emph{not} designed as a condition on the labeling of a graph. However, let $\Omega$ be a labeled graph with the $Gr_*'(\lambda)$--condition and let $R\subseteq F$ be the \emph{infinite} set  given by the labels of all cycles of $\Omega$. Then $R$ satisfies the Rips-Segev-$C'_{*}(\lambda)$--small cancellation condition. The proof of this claim is similar to the proof of Proposition \ref{P: criterion}. The converse statement is false.
\end{remark}


\section{Generalized Rips-Segev groups}

Our next aim is the construction of graphical presentations of torsion-free groups without the unique product property. Let us first give a definition of the unique product property.

\smallskip

Let $G$ be a group and let $A$ and $B$ be nonempty finite subsets of $F$. The  \emph{product of $A$ and $B$} is defined to be the set $AB=\{xy\mid x\in A,y\in B\}.$ 
If an element $z$ in $AB$ has a unique expression in $G$ as a product $z=xy$ for $x\in A$ and $y\in B$, then $A$ and $B$ are said to \emph{have a unique product in $G$}.

\begin{definition*}\label{up} If for all pairs of  nonempty finite subsets $A,B$ of $G$, the sets $A$ and $B$ have a unique product in $G$, then $G$ is said to have the \emph{unique product property} or to be a \emph{unique product group}.
\end{definition*}

In what follows we describe a procedure to construct groups without the unique product property, that is, admitting at least two nonempty finite subsets which do not have a unique product in the group.

\subsection{Graphs encoding the non-unique product property}

In this section we explain how to construct groups without the unique product property using graphical group presentations. We first consider an instructive example.

\smallskip

\begin{example}\label{E: instructive}
 Let $A=\{a,ab\}$ and $B=\{1,b\}$. 
The \emph{graphical presentation of $AB$} is the following graph. It can be seen as a subgraph of the Cayley graph of the free group on $a$ and $b$.  
 \[\xy
{\ar^b (-20,0)*{o}; (0,0)*{\bullet}}; {\ar^b (0,0); (20,0)*{o}}; (-20,-3)*{a\cdot 1};(0,-3)*{a\cdot b};(20,-3)*{ab\cdot b};(0,3)*{ab\cdot 1};
\endxy  
 \] 
Every vertex in this graph represents a product in $AB$: The two vertices marked by $o$ represent unique products in the free group. The vertex marked by $\bullet$ represents products which are not unique. As $a\cdot 1=a$ and $ab\cdot 1=ab$, for every element $x\in A$ there is a vertex $v$ representing $x$. Hence, for every product $xy$ in $AB$, $y\not=1$, there is a vertex $v$ representing $x$, a vertex $v'$ representing $xy$ and a simple path with label $y$ that starts at $v$  and terminates at $v'$. 

We now identify the vertices $o$. We obtain the following graph $\Omega$. 
\[\xy
{\ar@/_1pc/_b (-10,0)*{\bullet}; (10,0)}; {\ar@/_1pc/_b (10,0)*{\bullet}; (-10,0)}; (-15,2)*{ a\cdot 1}; (-15,-2)*{ab\cdot b};(15,-2)*{a\cdot b};(15.5,2)*{ab\cdot 1};
\endxy 
 \]
In $\Omega$, each vertex represents two different products in $AB$. The group $G(\Omega)$ is given by a presentation with generators $\{a,b\}$ and as relators the label on non-trivial cycles of $\Omega$. We view $G(\Omega)$ as the quotient of $\langle a\rangle *\langle b \rangle$ by the normal closure of $\{b^2\}$. The subsets $A$ and $B$ do not have a unique product in $G(\Omega)$. 

We have therefore constructed a graphical presentation of a group without the unique product property. 
\end{example}

The graph $\Omega$ of Example \ref{E: instructive} was used to \emph{encode the non-unique product property for $A$ and $B$}. Let us now turn to a more general case. Let $F$ denote the free product $G_1~*~\ldots~*~G_d$. Our aim is to extend and formalize the ideas used in Example \ref{E: instructive}. We  obtain groups in which more general sets $A, B\subset F$, $1\in B$, do not have a unique product.  Again, we encode the non-unique product property for $A$ and $B$ in a graph~$\Gamma$. We start with a graph $\Theta$ labeled by $F$ such that each vertex in $\Theta$ represents a product in $AB$. Moreover, for every product $xy$ in $AB$, $y\not=1$, there is a vertex $v$ representing $x$, a vertex $v'$ representing~$xy$ and a simple path with label $y$ that starts at $v$  and terminates at $v'$. If $v'$ represents a product that is unique in $F$, then we identify $v'$ with a second vertex $v''$ representing a different product in $AB$. Like this we produce our graph $\Gamma$. 
 Then we study conditions on $\Gamma$ that imply that $A$ and $B$ do not have a unique product in the corresponding group $G(\Gamma)$.
 Finally, we show that our constructions of graphs below yield groups without the unique product property.

\subsection{Generalized Rips-Segev graphs}

We now describe the construction of graphs encoding the non-unique product property for specific sets $A$ and $B$. We first describe the sets $A$ and $B$ and then construct the corresponding graphs in 4 steps. 

\smallskip

From now on let $G_1$ and $G_2$ be torsion-free groups with generating sets $X_1$ and $X_2$, and let $F=G_1*G_2$. Then $X:=X_1\sqcup X_2$ is a generating set for $F$. Let $1\not =a \in G_1$ and $1\not=b \in G_2$. We can always assume that $a$ and $b$ are in $X$. We now define the sets $A$ and $B$. Let us first look at an example.

   \begin{example}\label{E: RSset}
  Let $c$ { be a word on $X\sqcup X^{-1}$ which is weakly reduced and whose terminal letter does not coincide with $a^{-1}$, $b$ or $b^{-1}$, }and let $C\in \mathbb{N}$. Let $v_i:=ca^i $ and $w_i:=ca^ib$. Let $A:=\{v_0, \ldots, v_{C-1}\}$ and $B:=\{1,a,b,ab\}$. 
    The product set is given by $$AB=\{v_0,\ldots,v_C,w_0,\ldots,w_C\}.$$ The products $AB$ have the following graphical presentation $\Theta$. 
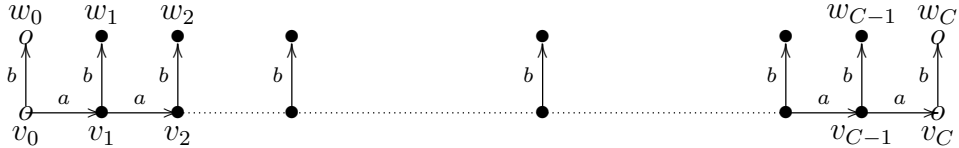
\[
\xy
 {\ar^a (0,0); (10,0)};(0,-3)*{v_0}; (10,-3)*{v_1}; 
 {\ar^a (10,0); (20,0)}; (20,-3)*{v_2};
 {(20,0)*{}; (100,0)**\dir{.}};{\ar^a (100,0); (110,0)}; (110,-3)*{v_{C-1}};
 {\ar^a (110,0); (120,0)}; (120,-3)*{v_C};
{\ar^b (0,0)*{o} ; (0,10)*{o}}; (0,13)*{w_0}; 
{\ar^b (10,0)*{\bullet}; (10,10)*{\bullet}}; (10,13)*{w_1};
{\ar^b (20,0)*{\bullet}; (20,10)*{\bullet}}; (20,13)*{w_2};
 {\ar^b (35,0)*{\bullet}; (35,10)*{\bullet}}; {\ar^b (68,0)*{\bullet}; (68,10)*{\bullet}}; 
{\ar^b (100,0)*{\bullet} ; (100,10)*{\bullet}};  {\ar^b (110,0)*{\bullet}; (110,10)*{\bullet}}; (110,13)*{w_{C-1}};
{\ar^b (120,0)*{o}; (120,10)*{o}}; (120,13)*{w_C};
\endxy
\]

The graph $\Theta$ can be seen as a subgraph of the Cayley graph of { $F$ with respect to $X\sqcup X^{-1}$}. 
Every vertex represent a product in $AB$. The elements $v_0,w_0,v_C,w_C$ are products in $AB$ that are unique in $F$. Every edge in this graph is directed and labeled by either $a$ or $b$.

By $T$, we denote the following subgraph of the Cayley graph of { $F$ with respect to~$X\sqcup X^{-1}$}.
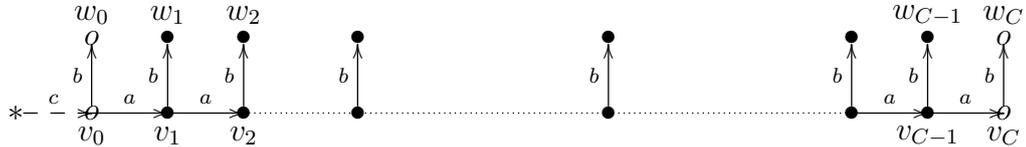
\[
\xy
{\ar@{-->}^{c} (-10,0)*{*}; (0,0)};
 {\ar^a (0,0); (10,0)};(0,-3)*{v_0}; (10,-3)*{v_1}; 
 {\ar^a (10,0); (20,0)}; (20,-3)*{v_2};
 {(20,0)*{}; (100,0)**\dir{.}};{\ar^a (100,0); (110,0)}; (110,-3)*{v_{C-1}};
 {\ar^a (110,0); (120,0)}; (120,-3)*{v_C};
{\ar^b (0,0)*{o} ; (0,10)*{o}}; (0,13)*{w_0}; 
{\ar^b (10,0)*{\bullet}; (10,10)*{\bullet}}; (10,13)*{w_1};
{\ar^b (20,0)*{\bullet}; (20,10)*{\bullet}}; (20,13)*{w_2};
 {\ar^b (35,0)*{\bullet}; (35,10)*{\bullet}}; {\ar^b (68,0)*{\bullet}; (68,10)*{\bullet}}; 
{\ar^b (100,0)*{\bullet} ; (100,10)*{\bullet}};  {\ar^b (110,0)*{\bullet}; (110,10)*{\bullet}}; (110,13)*{w_{C-1}};
{\ar^b (120,0)*{o}; (120,10)*{o}}; (120,13)*{w_C};
\endxy
\]
The additional vertex $*$ represents the identity, and the dashed arrow represents the path whose label is the word $c$. For every element $x$ in $A$ there is a vertex $v$ representing $x$ and a simple path $p_x$ from $*$ to $v$ that is labeled by $x$, and for every element $1\not= y\in B$, there is a vertex $v'$ representing $xy$ and a simple path $p_y$ from $v$ to the vertex $v'$ labeled by $y$. The concatenation $p_xp_y$ is the simple path from $*$ to $v'$ passing through $v$ that is labeled by $z:=xy$.
\end{example}

\smallskip

Now let us come back to the general case. Let $K$ be a non-zero  natural number. For $1\leqslant i \leqslant K$, let~$c_i$  be a weakly reduced word in $X\sqcup X^{-1}$, and let $C_i$  be  non-zero numbers in $\mathbb{N}$.   
     Let $v_{il}:=c_ia^l$ and $w_{il}=c_ia^lb$. We denote by $A$ the set 
 \[
 A:=\bigsqcup_{i=1}^K\{v_{i0},v_{i1},v_{i2},\ldots, v_{i(C_i-1)}\}
 \]
 and by $B$ the set $\{1,a,b,ab\}$. Then the product set 
 \[
 AB=\bigsqcup_{i=1}^K\{ v_{i0}, v_{i1}, \ldots, v_{iC_i},w_{i0},w_{i1}, \ldots, w_{iC_i}\}.
 \] 
 The number $K$, the elements $c_i$ and the numbers $C_i$ are thought of as of variables. In our explicit construction below, see Section \ref{S: explicit construction}, we determine possible values for them.

  \smallskip
  
 For each $i$ let $A_i:= \{v_{i0},v_{i1},\ldots,v_{i(C_i-1)}\}$ and let $\Theta_i$ be the graphical presentation of $A_iB$ as shown in Example \ref{E: RSset}. Let $\Theta:=\bigsqcup_{i=1}^K\Theta_i$ be the disjoint union of $K$ such graphs. Every edge of this graph is directed and labeled by either $a$ or $b$. Every vertex in $\Theta$ represents a product in $AB$. Moreover, for every element $x$ in $A$ there is a vertex $v$ representing $x$, and for every $1\not=y\in B$ there is a vertex $v'$ representing $xy$ and a simple path $p_y$ from $v$ to the vertex $v'$ labeled by $y$.
 
Let $p_i$ be the path whose label is the word $c_i$. Let us denote by $T'$ the graph obtained from $\Theta$ and the paths $p_i$ by  identifying the terminal vertex of $p_i$ with $v_{i0}$ for all $1\leqslant i \leqslant K$.  In $T'$ we identify the starting vertices of the paths $p_i$ and denote the  resulting tree by $T$. Let us denote the image of the starting vertices of the paths $p_i$ in $T$ by $*$. Again, for every element $x$ in $A$ there is a vertex $v$ representing $x$. In addition, in $T$ there is a simple path $p_x$ from $*$ to $v$ that is labeled by $x$, and for all  $1\not=y\in B$, there is a vertex $v'$ representing $xy$ and  simple path $p_y$ from $v$ to the vertex $v'$ labeled by $y$. The concatenation $p_xp_y$ is the simple path from $*$ to $v'$ passing through $v$ that is labeled by $xy$.
 
 We call the subgraph (in $\Theta$ or $T$) given by the vertices $\{v_{i0},v_{i1},\ldots,v_{iC_i}\}$ and edges $\{(v_{i0},v_{i1}),$ $(v_{i1},v_{i2}),$ $\ldots, (v_{i(C_i-1)},v_{iC_i})\}$ the \emph{$a$-line $i$}. For each $i$ there are four vertices of $\Theta_i$ (as a subgraph in $\Theta$ or $T$) representing a unique product in $F$, $v_{i0}$, $v_{iC_i}$, $w_{i0}$, and $w_{iC_i}$. To produce a graph that encodes the non-unique product property for $A$ and $B$ we copy the strategy of Example \ref{E: instructive} and identify every vertex representing a unique product with a vertex representing a different product.
 
\smallskip
 
 Starting with $\Theta$, we now construct graphs encoding the non-unique product property for $A$ and $B$ in 4 steps. In each step of the construction we call the image of the $a$-line $i$ again $a$-line $i$.

 \begin{enumerate}
  \item
For each of the $2n$ vertices $w$ of the form $w_{i0}$ and $w_{iC_i}$, we choose a vertex $v_{jI}$ such that $0\leqslant I \leqslant C_j$  and the pairs $(j,I)$ are all different among themselves. We then identify $w$ and $v_{jI}$, see 
Figure~\ref{Step 1}.

As two distinct vertices $w$ are not identified, the resulting graphs are reduced by construction. The vertices which have been identified  now represent at least two products of $AB$. We have possibly identified some vertices $w_{i0}$ or $w_{iC_i}$ with a vertex $v_{j0}$ or $v_{jC_j}$, cf. Figure~\ref{Step 1}. 
\item
 In a second step, we take care of vertices $v$ of the form $v_{j0}$ and $v_{jC_j}$ which have not been identified  with a vertex $w$ in Step 1. 
For every such vertex $v$ we choose a vertex $w_{lO}$, so that $0< O < C_l$ and the index pairs $(l,O)$ are all different among themselves. Then we identify $v$ and $w_{lO}$, see Figure~\ref{Step O}. 

Again, as two distinct vertices $v$ are not identified, the resulting graph is reduced by construction. In this graph, every vertex represents at least two products in $AB$.  It has exactly $K$  $a$-lines. We assume that the choices have been made so that the resulting graph is connected. It is easy to see that this is always possible.
\item
We delete all edges of degree one, i.e. all edges with a label by $b$ that have not been glued. 
 \end{enumerate}

 
   \begin{figure}[ht]
   \[
\xy
 {\ar^a (-10,0); (0,0)};(-10,-3)*{v_{i0}}; (0,-3)*{v_{i1}}; 
 {\ar^a (0,0); (10,0)}; (10,-3)*{v_{i2}};
 {(10,0)*{}; (90,0)**\dir{.}};{\ar^a (90,0); (100,0)}; (100,-3)*{v_{i(C_i-1)}};
 {\ar^a (100,0); (110,0)}; (110,-3)*{v_{iC_i}};
{\ar^b (-10,0)*{o} ; (-10,10)*{o}}; (-10,13)*{w_{i0}}; 
{\ar^b (0,0)*{\bullet}; (0,10)*{\bullet}}; (0,13)*{w_{i1}};
{\ar^b (10,0)*{\bullet}; (10,10)*{\bullet}}; (10,13)*{w_{i2}};
 {\ar^b (25,0)*{\bullet}; (25,10)*{\bullet}}; {\ar^b (58,0)*{\bullet}; (58,10)*{\bullet}}; 
{\ar^b (90,0)*{\bullet} ; (90,10)*{\bullet}};  {\ar^b (100,0)*{\bullet}; (100,10)*{\bullet}}; (100,13)*{w_{{i(C_i-1)}}};
{\ar@/_2pc/^b (110,0)*{o}; (98,20)}; (95,18)*{w_{iC_i}=v_{jI}};
 {\ar^a (20,20); (30,20)};(20,17)*{v_{j0}}; (30,17)*{v_{j1}}; 
 {\ar^a (30,20); (40,20)}; (40,17)*{v_{j2}};
 {(40,20)*{}; (120,20)**\dir{.}};{\ar^a (120,20); (130,20)}; (130,17)*{v_{j(C_j-1)}};
 {\ar^a (130,20); (140,20)}; (140,17)*{v_{jC_j}};
{\ar^b (20,20)*{o} ; (20,30)*{o}}; (20,33)*{w_{j0}}; 
{\ar^b (30,20)*{\bullet}; (30,30)*{\bullet}}; (30,33)*{w_{j1}};
{\ar^b (30,20)*{\bullet}; (30,30)*{\bullet}}; 
 {\ar^b (55,20)*{\bullet}; (55,30)*{\bullet}}; {\ar^b (98,20)*{\bullet}; (98,30)*{\bullet}}; 
{\ar^b (120,20)*{\bullet} ; (120,30)*{\bullet}};  {\ar^b (130,20)*{\bullet}; (130,30)*{\bullet}}; (130,33)*{w_{j(C_j-1)}};
{\ar^b (140,20)*{o}; (140,30)*{o}}; (140,33)*{w_{jC_j}};
\endxy
\]
   \[
\xy
 {\ar^a (-10,0); (0,0)};(-10,-3)*{v_{i0}}; (0,-3)*{v_{i1}}; 
 {\ar^a (0,0); (10,0)}; (10,-3)*{v_{i2}};
 {(10,0)*{}; (90,0)**\dir{.}};{\ar^a (90,0); (100,0)}; (100,-3)*{v_{i(C_i-1)}};
 {\ar^a (100,0); (110,0)}; (110,-3)*{v_{iC_i}};
{\ar^b (-10,0)*{o} ; (-10,10)*{o}}; (-10,13)*{w_{i0}}; 
{\ar^b (0,0)*{\bullet}; (0,10)*{\bullet}}; (0,13)*{w_{i1}};
{\ar^b (10,0)*{\bullet}; (10,10)*{\bullet}}; (10,13)*{w_{i2}};
 {\ar^b (25,0)*{\bullet}; (25,10)*{\bullet}}; {\ar^b (58,0)*{\bullet}; (58,10)*{\bullet}}; 
{\ar^b (90,0)*{\bullet} ; (90,10)*{\bullet}};  {\ar^b (100,0)*{\bullet}; (100,10)*{\bullet}}; (100,13)*{w_{{i(C_i-1)}}};
{\ar@/_1pc/^b (110,0)*{o}; (120,20)}; (130,18)*{w_{iC_i}=v_{jC_j}};
 {\ar^a (0,20); (10,20)};(0,17)*{v_{j0}}; (10,17)*{v_{j1}}; 
 {\ar^a (10,20); (20,20)}; 
 {(20,20)*{}; (100,20)**\dir{.}};{\ar^a (100,20); (110,20)}; (110,17)*{v_{j(C_j-1)}};
 {\ar^a (110,20); (120,20)}; 
{\ar^b (0,20)*{o} ; (0,30)*{o}}; (0,33)*{w_{j0}}; 
{\ar^b (10,20)*{\bullet}; (10,30)*{\bullet}}; (10,33)*{w_{j1}};
{\ar^b (10,20)*{\bullet}; (10,30)*{\bullet}}; 
 {\ar^b (35,20)*{\bullet}; (35,30)*{\bullet}}; {\ar^b (78,20)*{\bullet}; (78,30)*{\bullet}}; 
{\ar^b (100,20)*{\bullet} ; (100,30)*{\bullet}};  {\ar^b (110,20)*{\bullet}; (110,30)*{\bullet}}; (110,33)*{w_{j(C_j-1)}};
{\ar^b (120,20)*{\bullet}; (120,30)*{o}}; (120,33)*{w_{jC_j}};
\endxy
\] 
\caption{Step 1} \label{Step 1}
\end{figure} 
\begin{figure}[ht!]
   \[
\xy
 {\ar^a (-10,0); (0,0)};(-10,-3)*{v_{l0}}; (0,-3)*{v_{l1}}; 
 {\ar^a (0,0); (10,0)}; (10,-3)*{v_{l2}};
 {(10,0)*{}; (90,0)**\dir{.}};{\ar^a (90,0); (100,0)}; (100,-3)*{v_{l(C_l-1)}};
 {\ar^a (100,0); (110,0)}; (110,-3)*{v_{lC_l}};
{\ar^b (-10,0)*{o} ; (-10,10)*{\bullet}}; (-10,13)*{w_{l0}}; 
{\ar^b (0,0)*{\bullet}; (0,10)*{\bullet}}; (0,13)*{w_{l1}};
{\ar^b (10,0)*{\bullet}; (10,10)*{\bullet}}; (10,13)*{w_{l2}};
 {\ar^b (25,0)*{\bullet}; (20,20)*{\bullet}}; (25,-3)*{v_{lO}};
 {\ar^b (58,0)*{\bullet}; (58,10)*{\bullet}}; 
{\ar^b (90,0)*{\bullet} ; (90,10)*{\bullet}};  {\ar^b (100,0)*{\bullet}; (100,10)*{\bullet}}; (100,13)*{w_{{l(C_l-1)}}};
{\ar^b (110,0)*{o}; (110,10)*{\bullet}}; (110,13)*{w_{lC_l}};
 {\ar^a (20,20); (30,20)};(12,17)*{w_{lO}=v_{j0}}; (30,17)*{v_{j1}}; 
 {\ar^a (30,20); (40,20)}; (40,17)*{v_{j2}};
 {(40,20)*{}; (120,20)**\dir{.}};{\ar^a (120,20); (130,20)}; (130,17)*{v_{j(C_j-1)}};
 {\ar^a (130,20); (140,20)}; (140,17)*{v_{jC_j}};
{\ar^b (20,20)*{o} ; (20,30)*{\bullet}}; (20,33)*{w_{j0}}; 
{\ar^b (30,20)*{\bullet}; (30,30)*{\bullet}}; (30,33)*{w_{j1}};
{\ar^b (30,20)*{\bullet}; (30,30)*{\bullet}};
 {\ar^b (55,20)*{\bullet}; (55,30)*{\bullet}}; {\ar^b (98,20)*{\bullet}; (98,30)*{\bullet}}; 
{\ar^b (120,20)*{\bullet} ; (120,30)*{\bullet}};  {\ar^b (130,20)*{\bullet}; (130,30)*{\bullet}}; (130,33)*{w_{j(C_j-1)}};
{\ar^b (140,20)*{o}; (140,30)*{\bullet}}; (140,33)*{w_{jC_j}};
\endxy
\]
\caption{Step 2} \label{Step O}
\end{figure}

Let us denote the resulting graph  by $\Theta'$. By construction the labeling of $\Theta'$ is reduced over~$F$. We refer to a vertex by $v_{iP}$ if it is the image of $v_{iP}$ of $\Theta$. The $a$-line $i$ in $\Theta'$ consists of the vertices $\{v_{i0},v_{i1},\ldots,v_{iC_i}\}$ and edges $\{(v_{i0},v_{i1}),(v_{i1},v_{i2}),\ldots ,(v_{i(C_i-1)},v_{iC_i})\}$. Each $a$-line $i$ has $s_i$ distinguished vertices $v_{iI_{i1}},\ldots,v_{iI_{is_i}}$  arising from Step 1, and $t_i$ distinguished vertices $v_{iO_{i1}},\ldots,v_{iO_{it_i}}$ arising from Step 2. Both $s_i$ and $t_i$ can be zero.  Figure \ref{F: typical aline} shows a typical $a$-line in $\Theta'$, Figures \ref{F: extremal case 1} and \ref{F: extremal case 2} show further examples of a-lines that can be produced using the above construction.

\smallskip\medskip

\emph{
 Let us denote the tuple $(I_{i1},\ldots,I_{is_i},O_{i1},\ldots,O_{it_i},C_i)$  by $(I_i,O_i,C_i)$. 
Then,  we denote  
\[
(I,O,C):=\bigcup_{1\leqslant i\leqslant K}\{(i,s_i,t_i,(I_i,O_i,C_i))\}.
\] 
We call the graph $\Theta'$ constructed above a generalized Rips-Segev graph with coefficients $(I,O,C)$.
}

\smallskip\medskip

These intermediate graphs $\Theta'$ turn out to be useful later on. 
In a fourth step we now make $\Theta'$ a graph encoding the non-unique product property for $A$ and $B$.
\begin{enumerate}\setcounter{enumi}{3}
\item For each word $c_i$  we take the labeled path $p_i$ whose label is the word $c_i$. For all $i$, we glue the terminal vertex of the path $p_i$ to $v_{i0}$ in $\Theta'$. We call the graph so obtained $\Theta''$. In $\Theta''$ we identify all the starting vertices of the paths $c_i$. The reduction (or folding) of $\Theta''$ is denoted by $\Gamma$.
\end{enumerate}

\smallskip

\emph{
We call a graph $\Gamma$ with a reduced labeling {by $F$}
, $K\in \mathbb{N}$, coefficients $(I,O,C)$ and elements $c_i\in F$ constructed above a generalized Rips-Segev graph for $A$ and $B$.
}

\smallskip\medskip

Figure \ref{F: exnup4} shows two examples of generalized Rips-Segev graphs. 
\smallskip

We note that $\Gamma$ contains the image of the tree $T$ as a subgraph. More specifically, the tree $T$ is a maximal spanning tree of $\Gamma$.  Let us remark that the graph $\Gamma$ could alternatively be obtained by applying steps 1, 2, and 3 to the tree $T$ defined above. In the next sections  however, we work with the intermediate graphs $\Theta'$, that we called Rips-Segev graphs for coefficients $(I,O,C)$.

\begin{figure}[ht!]
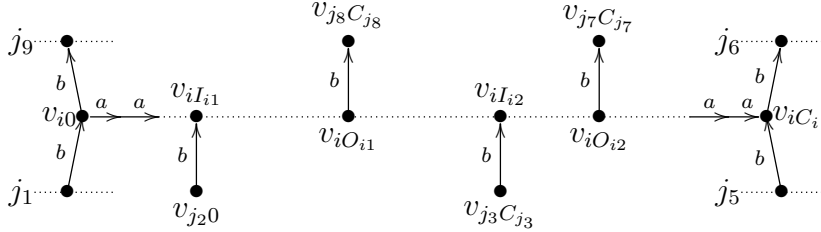

 \[
\xy
 {\ar^a (0,0)*{\bullet}; (5,0)};(-3,0)*{v_{i0}}; {\ar^a (5,0); (10,0)};{(10,0)*{}; (80,0)**\dir{.}};{\ar^a (80,0); (85,0)};{\ar^a (85,0); (90,0)*{\bullet}}; (94,0)*{v_{iC_i}};
{\ar^b (15,-10)*{\bullet} ; (15,0)*{\bullet}};  (15,3)*{v_{iI_{i1}}}; (15,-13)*{v_{j_20}};
 {\ar^b (55,-10)*{\bullet}; (55,0)*{\bullet}};  (55,3)*{v_{iI_{i2}}}; (55,-13)*{v_{j_3C_{j_3}}};
 {\ar^b (35,0)*{\bullet}; (35,10)*{\bullet}}; (35,-3)*{v_{iO_{i1}}}; (35,13.5)*{v_{j_8C_{j_8}}};
{\ar^b (68,0)*{\bullet}; (68,10)*{\bullet}}; (68,-3)*{v_{iO_{i2}}}; (68,13)*{v_{j_7C_{j_7}}};
{\ar^b (-2,-10)*{\bullet} ; (0,0)};  {\ar^b (92,-10)*{\bullet}; (90,0)}; {\ar@{.} (-8,-10)*{j_1}; (4,-10)}; {\ar@{.} (85,-10)*{j_5}; (98,-10)};
{\ar^b (0,0); (-2,10)*{\bullet}}; {\ar^b (90,0); (92,10)*{\bullet}}; {\ar@{.} (-8,10)*{j_9}; (4,10)}; {\ar@{.} (85,10)*{j_6}; (98,10)};
\endxy
\]
\caption{Typical $a$-line with $s=t=2$.} \label{F: typical aline}
\end{figure}
\begin{figure}[ht!]
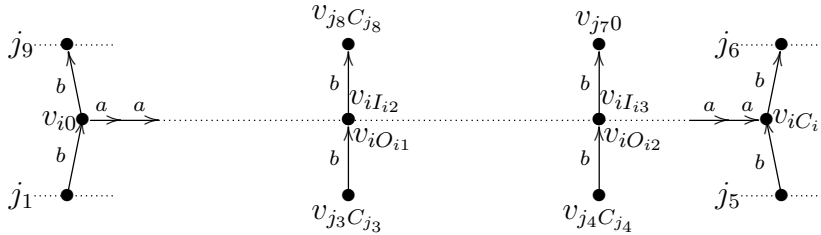

\[
\xy
 {\ar^a (0,0)*{\bullet}; (5,0)};(-3,0)*{v_{i0}}; {\ar^a (5,0); (10,0)};{(10,0)*{}; (80,0)**\dir{.}};{\ar^a (80,0); (85,0)};{\ar^a (85,0); (90,0)*{\bullet}}; (94,0)*{v_{iC_i}};
 {\ar^b (35,-10)*{\bullet}; (35,0)*{\bullet}};  (38.5,2.5)*{v_{iI_{i2}}}; (35,-13)*{v_{j_3C_{j_3}}};
 {\ar^b (68,-10)*{\bullet}; (68,0)*{\bullet}};  (71.5,2.5)*{v_{iI_{i3}}}; (68,-13)*{v_{j_4C_{j_4}}};
 {\ar^b (35,0)*{\bullet}; (35,10)*{\bullet}}; (39.5,-2.5)*{v_{iO_{i1}}}; (35,13.5)*{v_{j_8C_{j_8}}};
{\ar^b (68,0)*{\bullet}; (68,10)*{\bullet}}; (72.5,-2.5)*{v_{iO_{i2}}}; (68,13)*{v_{j_70}};
{\ar^b (-2,-10)*{\bullet} ; (0,0)};  {\ar^b (92,-10)*{\bullet}; (90,0)}; {\ar@{.} (-8,-10)*{j_1}; (4,-10)}; {\ar@{.} (85,-10)*{j_5}; (98,-10)};
{\ar^b (0,0); (-2,10)*{\bullet}}; {\ar^b (90,0); (92,10)*{\bullet}}; {\ar@{.} (-8,10)*{j_9}; (4,10)}; {\ar@{.} (85,10)*{j_6}; (98,10)};
\endxy
\]
\caption{a-line with $I_{jl}=O_{jk}$ for some $k,l$.} \label{F: extremal case 1}
\end{figure}
 \begin{figure}[ht!]
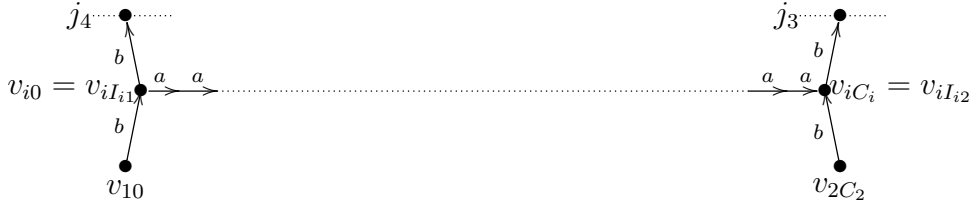

\[
\xy
 {\ar^a (0,0)*{\bullet}; (5,0)};(-9,0)*{v_{i0}=v_{iI_{i1}}}; {\ar^a (5,0); (10,0)};{(10,0)*{}; (80,0)**\dir{.}};{\ar^a (80,0); (85,0)};{\ar^a (85,0); (90,0)*{\bullet}}; (100,0)*{v_{iC_i}=v_{iI_{i2}}}; 
{\ar^b (-2,-10)*{\bullet} ; (0,0)};  {\ar^b (92,-10)*{\bullet}; (90,0)}; (-2,-13)*{v_{10}}; (92,-13)*{v_{2C_2}};
{\ar^b (0,0); (-2,10)*{\bullet}}; {\ar^b (90,0); (92,10)*{\bullet}}; {\ar@{.} (-8,10)*{j_4}; (4,10)}; {\ar@{.} (85,10)*{j_3}; (98,10)};
\endxy
\]
\caption{a-line with $s=2$, $t=0$, $I_{i1}=0$ and $I_{i2}=C_i$.} \label{F: extremal case 2}
\end{figure}

 \begin{figure}[ht!]
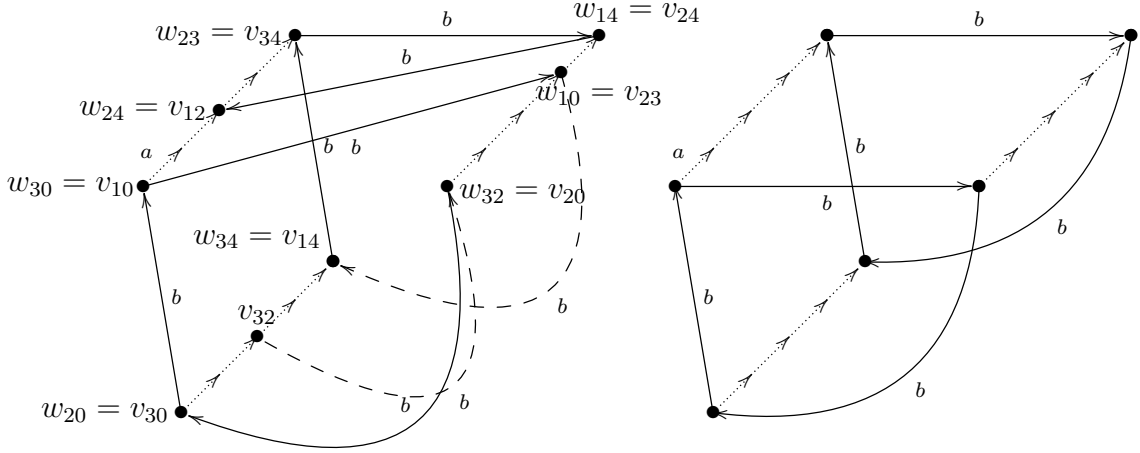

           \[
\xy
{\ar@{.>}^a (-10,5)*{}; (-5,10)}; {\ar_b (-10,5); (45,20)*{\bullet}}; (50,17)*{w_{10}=v_{23}};
{\ar@/^{60pt}/@{-->}^b (45,20);(15,-5)*{\bullet}}; (-19.5,5)*{w_{30}=v_{10}}; 
(-10,15)*{w_{24}=v_{12}}; 
{\ar@{.>} (-5,10)*{}; (0,15)*{}}; 
{\ar@{.>} (0,15)*{}; (5,20)*{}}; 
{\ar@{.>} (5,20)*{}; (10,25)*{}}; {\ar^b (10,25); (50,25)*{\bullet}}; (55,28)*{w_{14}=v_{24}}; 
{\ar@{.>} (30,5)*{}; (35,10)*{}}; {\ar@/^5pc/_b (30,5); (-5,-25)*{\bullet}} ; (40,4)*{w_{32}=v_{20}}; (-15,-25)*{w_{20}=v_{30}};
{\ar@{.>} (35,10)*{}; (40,15)*{}};
{\ar@{.>} (40,15)*{}; (45,20)*{}}; 
{\ar@{.>} (45,20)*{}; (50,25)*{}}; {\ar_b (50,25); (0,15)*{\bullet}}; 
{\ar@{.>} (-5,-25)*{}; (0,-20)*{}}; {\ar_b (-5,-25); (-10,5)*{\bullet}}; 
{\ar@{.>} (0,-20)*{}; (5,-15)*{}}; {\ar@/_5pc/@{-->}_b (5,-15)*{\bullet}; (30,5)*{\bullet}}; (5,-12)*{v_{32}};
{\ar@{.>} (5,-15)*{}; (10,-10)*{}}; 
{\ar@{.>} (10,-10)*{}; (15,-5)*{\bullet}}; {\ar_b (15,-5); (10,25)*{\bullet}}; (0,25)*{w_{23}=v_{34}}; (5,-2)*{w_{34}=v_{14}};
{\ar@{.>}^a (60,5)*{\bullet}; (65,10)}; {\ar_b (60,5); (100,5)*{\bullet}};{\ar@/^{30pt}/^b (120,25);(85,-5)}; 
{\ar@{.>} (65,10)*{}; (70,15)*{}}; 
{\ar@{.>} (70,15)*{}; (75,20)*{}}; 
{\ar@{.>} (75,20)*{}; (80,25)*{\bullet}}; {\ar^b (80,25); (120,25)*{\bullet}}; 
{\ar@{.>} (100,5)*{\bullet}; (105,10)*{}}; {\ar@/^3pc/^b (100,5); (65,-25)*{}}; 
{\ar@{.>} (105,10)*{}; (110,15)*{}};
{\ar@{.>} (110,15)*{}; (115,20)*{}}; 
{\ar@{.>} (115,20)*{}; (120,25)*{\bullet}};  
{\ar@{.>} (65,-25)*{\bullet}; (70,-20)*{}}; {\ar_b (65,-25); (60,5)*{}};
{\ar@{.>} (70,-20)*{}; (75,-15)*{}}; 
{\ar@{.>} (75,-15)*{}; (80,-10)*{}}; 
{\ar@{.>} (80,-10)*{}; (85,-5)*{\bullet}}; {\ar_b (85,-5); (80,25)*{\bullet}};
\endxy
\]
\caption{Generalized Rips-Segev graphs for $A$ and $B$, where $K=3$, $C_1=C_2=C_3=4$, and $c_1=1$, $c_2=b^{-1}$, and $c_3=b^{-2}$. Both graphs are results of our above construction starting with this given data and their labeling is reduced. The full $b$-edges are those glued in Step 1 above. The dashed $b$-edges are those glued in Step 2 above. On the right-hand side no $b$-edges are glued in Step~2. }   
\label{F: exnup4}
\end{figure} 

\newpage


We now discuss conditions to use $\Gamma$, as the graph $\Omega$ in Examples \ref{E: instructive}, to encode the non-unique product property for $A$ and $B$. Indeed, let $G(\Gamma)$ be the quotient of $F$ by the normal subgroup generated by the set of labels on the non-trivial cycles of $\Gamma$. 
\begin{prop}\label{P: small cancellation on RSgraphs} 
Suppose $\Gamma$ is a generalized Rips-Segev graph for $A$ and $B$ whose labeling satisfies the $Gr_*'(1/6)$--condition. Then $A$ and $B$ embed and  do not have a unique product in $G(\Gamma)$.
\end{prop}

The proposition is implied by  the following two Lemmas. Denote by $\iota$ the quotient map $F\to G(\Gamma)$.
\begin{lemma}\label{L: graphs encoding non-unique product property}
 Let $\Gamma$ be a generalized Rips-Segev graph for $A$ and $B$. Suppose that the map $\iota$ is injective when restricted to $A$ and $B$. Then $A$ and $B$ do not have a unique product in $G(\Gamma)$.
\end{lemma}
\begin{proof}
By construction $\Gamma$ contains the image of the tree $T$. Hence, for every element $x\in A$ there is a vertex $v\in T$ such that the simple path $p_x$ from $*$ to $v$ is labeled by $x$. For every product $z=xy$ in $AB$ there is a simple path $p_y$ from $v$ to a vertex $v'$ labeled by $y$. The concatenation $p_xp_y$ is a path from $*$ to $v'$ that is labeled by $z=xy$. If $z$ is unique in $F$, then $v'$ was identified in Step 1 or 2 with a different vertex $v''$ of $T$. Let  $p'$ be the simple path from $*$ to $v''$ in $T$. By construction, the label $p'$ equals a product $z'=x'y'$ in $AB$ where $x\not=x'$ and $y\not=y'$ in $F$. The concatenation of $p$ and $p'^{-1}$ is a cycle in $\Gamma$, and the label of this cycle equals $zz'^{-1}=xy(x'y')^{-1}$.  As $\iota$ is injective when restricted to $A$ and $B$, $x\not= x'$ and $y\not= y'$ in $G(\Gamma)$. This implies that $z$ is not a unique product of $A$ and $B$ in $G(\Gamma)$.
\end{proof}

The following lemma completes the proof of Proposition \ref{P: small cancellation on RSgraphs}.

\begin{lemma}
Let $\Gamma$ be a generalized Rips-Segev graph for $A$ and $B$ whose labeling satisfies the $Gr_*'(1/6)$--condition. Then $A$ and $B$ embed into the group $G(\Gamma)$.
\end{lemma}

\begin{proof}
By construction, $\Gamma$ is connected. 
 Let us show that $\iota$ is injective when restricted to $A$: Suppose that $x_1\not= x_2\in A$ and $\iota(x_1)=\iota(x_2)$ in~$G(\Gamma)$. By construction, there are  vertices $v_1$ and $v_2$ in $\Gamma$, $v_1\not= v_2$, and paths $p_{x_1}$ and $p_{x_2}$ from $*$ to $v_1$ and $v_2$ respectively whose labels   represent $x_1$, respectively $x_2$. Let $p$ be the path obtained by concatenating the inverse of the path $p_{x_1}$ with $p_{x_2}$: The path $p$ starts at $v_1$ goes back to $*$ along the edges of $p_{x_1}$ in reverse order, and then from $*$ to $v_2$ along $p_{x_2}$. 
 
 Let $\mathscr{C}$ be the Cayley graph of $G$ with respect to $X\sqcup X^{-1}$. By Theorem \ref{T: gi}, the graph $\Gamma$ injects into $\mathscr{C}$. Hence $v_1\not= v_2$ in $\mathscr{C}$, and by consequence the image of $p$ in $\mathscr{C}$ is not a cycle. Hence,  $\iota(x_1)\not=\iota(x_2)$ 
 in $G$.  
 
A similar argument can be applied to  conclude that $\iota$ is injective when restricted to $B$.
\end{proof}

Under the $Gr_*(1/8)$--condition, Theorem \ref{T: tf} implies that $G(\Gamma)$ is torsion-free. Theorem \ref{lii} implies that $G(\Gamma)$ is Gromov hyperbolic if $G_1$ and $G_2$ are Gromov hyperbolic.

\subsection{Graphical small cancellation}\label{S: RipsSegev condition}

The labellings of the graphs shown in Figure~\ref{F: exnup4} do not satisfy the $Gr_*'(1/8)$--condition. We now investigate conditions on generalized Rips-Segev graphs~$\Gamma$ that ensure this condition.

\smallskip

Let us start with a remark. Consider a  generalized Rips-Segev graph $\Theta'$ for coefficients $(I,O,C)$. We specify $c_i$ to be the label of one of the shortest paths connecting $v_{10}$ and $v_{i0}$ in $\Theta'$. This specifies a set $A$. Observe that the corresponding generalized Rips-Segev graph $\Gamma$ for this set $A$ and the set $B$ equals the graph $\Theta'$, see Figure~\ref{F: exnup4} for illustration. From now we choose $c_i$ as above and assume that $\Gamma$ is a generalized Rips-Segev graph with coefficients $(I,O,C)$. 
\smallskip

We distinguish two types of directed edges in $\Gamma$ : \begin{itemize}
                                                                                                \item $a$-edges, that is positively oriented edges labeled $a$
                                                                                                \item $b$-edges, that is positively oriented edges labeled $b$.
                                                                                               \end{itemize}
Every $b$-edge is by construction an edge $(v_{iP},v_{jQ})$ where $i\not=j$, $1\leqslant i,j \leqslant K$, $P\in \{0,O_{i1},\ldots,O_{it_i},C_i\}$, and $Q\in \{I_{j1},\ldots, I_{js_j}\}$.
                                                                                               
\begin{definition*}
 The \emph{underlying graph of $\Gamma$} is the following graph denoted by $\Phi_{\Gamma}$. 
 
 The vertex set of $\Phi_{\Gamma}$ is the set of a-lines in $\Gamma$. We denote these vertices by $\{1,2,3,\ldots, K\}$, so that vertex $i$ represents the $a$-line $i$. 
 
 The edges of $\Phi_{\Gamma}$   are in one-one correspondence with the $b$-edges of $\Gamma$: For every $b$-edge $(v_{iP},v_{jQ})$, there is a single directed edge $(i,j)$ in $\Phi_{\Gamma}$ . Conversely, for every edge $e=(i,j)$ in $\Phi_{\Gamma}$  there is a unique number  $P_e\in \{0,O_{i1},\ldots,O_{it_i},C_i\}$ and a unique number $Q_e\in \{I_{j1},\ldots, I_{js_j}\}$ such that $(v_{iP_e},v_{jQ_e})$ is a $b$-edge in $\Gamma$.

 If $(v_{iP},v_{jQ})$ is a $b$-edge in $\Gamma$, then we label the corresponding edge $(i,j)$ in $\Phi_{\Gamma}$  by $a^{P}ba^{-Q}$.
  \end{definition*}
  
  We now describe a reverse construction.
  \begin{definition*}
 We refer to the following reduction procedure as of \emph{$\{a,b\}$--reduction} of a graph $\Phi$ whose edges are labeled by words of the type $a^{P}ba^{-Q}$, $P,Q\in \mathbb{N}$. Suppose that the edge $(i,j)$ is labeled by $a^{P_i}ba^{-Q_j}$. We replace the edge $(i,j)$ by a path $p$ of $P_i+Q_j+1$ edges in same orientation as $(i,j)$. Then we label the first $P_i$ edges by $a$, the $(P_i+1)$-th edge by $b$ and the remaining $Q_j$ edges by $a^{-1}$ so that the label on  $p$ is the word $a^{P_i}ba^{-Q_j}$.  We denote by $\Phi'$ the graph obtained by applying this procedure to all edges of $\Phi$. Finally, we  reduce (or fold) the graph $\Phi'$ and obtain a reduced graph $\Phi''$ labeled by $\{a^{\pm1},b^{\pm1}\}$. The reduced graph $\Phi''$ is the $\{a,b\}$--reduction of $\Phi$.  
 \end{definition*}
 
 The following proposition is immediate.
 \begin{prop}The $\{a,b\}$--reduction of $\Phi_{\Gamma}$ coincides with the generalized Rips-Segev graph $\Gamma$.  
 \end{prop}


  Let us analyze $(I,O,C)$ and $\Phi_{\Gamma}$ to imply the $Gr_*'(1/8)$--condition on the labeling of~$\Gamma$.
 
{

\smallskip
 
 We employ the following two assumptions.  Firstly, we assume for simplicity that $t_i=s_i=2$ for all a-lines $i$,  and write \[(I_i,O_i,C_i)=(I_{i1},I_{i2},O_{i1},O_{i2},C_i). \] 
Secondly, we assume that $(I,O,C)$ satisfies the following Rips-Segev condition. 
\begin{definition*} A set of coefficients $(I,O,C)$ satisfies the \emph{Rips-Segev condition} if the numbers in the list 
 \begin{align*} 
I_{11},\ldots, I_{i1},&I_{i2},O_{i1},O_{i2},C_i,|C_i-O_{i2}|,|C_i-O_{i1}|,|C_i-I_{i2}|,|C_i-I_{i1}|,\\ &|O_{i2}-O_{i1}|,|O_{i2}-I_{i2}|,|O_{i2}-I_{i1}|,|O_{i1}-I_{i2}|,|O_{i1}-I_{i2}|,|I_{i2}-I_{i1}|,\ldots
\end{align*} are all non-zero and pairwise distinct.
\end{definition*}
For an (infinite) set of tuples $(I_i,O_i,C_i)$ that satisfy such conditions see Example \ref{E:coefM=0} below.}

\smallskip

The first condition above implies that the degree of the vertices of the underlying graph equals 8. 
The numbers in the Rips-Segev condition are all distances (in the edge-distance) between pairs of the vertices $v_{i0},$ $v_{iI_{ij}},$ $v_{iO_{ik}}$ and $v_{iC_i}$ on an $a$-line. The Rips-Segev condition then states, in other words, that all distances on the $a$-lines between pairs of those distinguished vertices are distinct among each other. These vertices are the starting or terminal vertices of the $b$-edges. { The following remarks follow immediately from the construction.

\begin{remark}\label{R: RScondition 1}
Let $b^{\varepsilon_1}a^{P}b^{\varepsilon_2}$, $\varepsilon_l=\pm1$, $P\in \mathbb{Z}$, be represented by the label of a simple path in  $\Gamma$. If $P$ is non-zero, then $P$ is one of the numbers in the Rips-Segev condition. The Rips-Segev condition then implies that $P$ is distinct among all non-zero exponents of $a$ in such labels of simple paths. 
\end{remark}
\begin{remark} \label{R: RScondition 2}
If $b^{\varepsilon_1}a^{P}b^{2\varepsilon_2}$, respectively $b^{2\varepsilon_1}a^{P}b^{\varepsilon_2}$, is represented by the label of  a simple  path in~$\Gamma$, the Rips-Segev condition implies that $P$ is non-zero.
\end{remark}

As a consequence of Remark \ref{R: RScondition 1} }, for all $P\not=0$, $\varepsilon_l=\pm1$, no graphical piece in  $\Gamma$ contains subpath whose label is equal to~$b^{\varepsilon_1}a^{P}b^{\varepsilon_2}$.

\begin{lemma}\label{L: maxplength}
Under the Rips-Segev condition, for all graphical pieces $p$ in $\Gamma$, we have that $\lvert \omega(p) \rvert_* \leqslant 3$.
\end{lemma}

We estimate the  \emph{minimal cycle length} $\gamma(\Gamma):=\min\{|\omega(c)|_*\mid c \text{ is a non-trivial cycle in $\Gamma$} \}$. 
We estimate $\gamma(\Gamma)$ in terms of the minimal number of $b$-edges on non-trivial cycles in $\Gamma$. Therefore, we estimate $\gamma(\Gamma)$  using the underlying graph of $\Gamma$.

\begin{lemma}\label{L: mincycle}
Under the Rips-Segev condition, we have that $ \gamma(\Gamma) \geqslant \girth (\Phi_{\Gamma}).$ 

 The girth denotes the minimal number of edges in the cycles of a graph.
\end{lemma}
\begin{proof} Let $g:=\girth (\Phi_{\Gamma}).$ { Let  
\[
 w=a^{P_0}b^{\varepsilon_0}a^{P_1}b^{\varepsilon_1}a^{P_2} \ldots a^{P_{l-1}}b^{\varepsilon_{l-1}},\hbox{ $P_i\in \mathbb{Z} \cup \{0\}$, $\varepsilon_i=\pm1$, $0\leqslant i < l$}
\]
 be the label of a  simple cycle in $\Gamma$. Remark \ref{R: RScondition 2} implies: For all $1\leqslant i < l$, if $P_{i}=0$, then $P_{i-1}\not=0$ and $P_{i+1}\not= 0$ . If $P_0=0$, then $P_1\not= 0$ and $P_{l-1}\not= 0$. Hence, the label of a non-trivial cycle in $\Gamma$ is of minimal free product length if it is representing 
\begin{align*}
 w&=a^{P_0}b^{2\varepsilon_0}a^{P_1}b^{2\varepsilon_1}a^{P_2} \ldots a^{P_{k}}b^{2\varepsilon_{k}}\hbox{ if $l$ is even and $k:=l/2$, and }          \\
  w&=a^{P_0}b^{2\varepsilon_0}a^{P_1}b^{2\varepsilon_1}a^{P_2} \ldots a^{P_{k-1}}b^{2\varepsilon_{k-1}}a^{P_{k}}b^{\varepsilon_{k}}\hbox{ if $l$ is odd and $k:=\lfloor l/2 \rfloor +1$.}    
 \end{align*}
By construction, $\varepsilon_i=\pm 1$, and by the Rips-Segev condition all the exponents $P_j$, $0\leqslant j \leqslant k$, of $w$ are non-zero. In both cases, we have that $|w|_*=2k+2>l+2$. Note that the non-trivial cycles in $\Gamma$ are in one-one correspondence with the cycles in $\Phi_{\Gamma}$. In particular, a non-trivial cycle in $\Gamma$ contains at least $g$ distinct  $b$-edges. Hence, set $l=g-2$. We now conclude that a label of a non-trivial cycle in $\Gamma$ has length at least $g$. This proves the claim.}
\end{proof}

Using the criterion described in Proposition \ref{P: criterion}, we obtain the following.

\begin{cor}\label{C: final graphical small cancellation for Rips-Segev}
A generalized Rips-Segev graph with coefficients $(I,O,C)$ satisfying the Rips-Segev condition and with underlying graph of girth $41$ satisfies the $Gr_*(1/8)$--condition.
\end{cor}
 The groups defined by such graphs are torsion-free and without the unique product property. We comment on the existence of such generalized Rips-Segev graphs in the next section.

\subsection{Explicit constructions}\label{S: explicit construction}

In this section, we label a specific graph with large girth by \emph{words} in $a^{\pm1}$ and $b^{\pm1}$ such that the $\{a,b\}$--reduction of this graph is a generalized Rips-Segev graph $\Gamma$. Then we combine such graphs to obtain new families of graphs that define torsion-free groups without the unique product property.  We work with the following explicit coefficients.

\begin{example}\label{E:coefM=0} The following coefficients satisfy the Rips-Segev condition.
 \[
		\begin{array}{c|c|c|c|c|c|}
			   j & I_{j1}  & I_{j2} & O_{j1}  & O_{j2} & C_j \\  \hline 
                           1 &  10 &  100 & 1000 & 10000 & 10^5  \\
			   2 & 10^6 & 10^7 & 10^8 & 10^9 & 10^{10}  \\
			    \hdots &  \hdots  & \hdots& \hdots & \hdots & \hdots \\
			   & & & & & \\
			   i & 10^{5i-4} & 10^{5i-3} & 10^{5i-2} & 10^{5i-1} & 10^{5i}\\  
			    \hdots & \hdots &  \hdots  & \hdots& \hdots & \hdots  \\
			   & & & & & \\
                  \end{array}
\]
\end{example}

Other more technical such examples are explained in \cite{rips_torsion-free_1987}*{p. 125f}.  Choosing an injective map $\varphi:\{1,\ldots,n\} \to \mathbb{N}$, we obtain a set of coefficients $(I,O,C)$ by setting $(I_{i1},I_{i2},O_{i1},O_{i2},C_i)$ the $\varphi(i)$-th line in Example \ref{E:coefM=0}. Using such coefficients it is clear that we can construct $\Gamma$ to satisfy Lemma \ref{L: maxplength} above. The purpose of the remaining section is to ensure the large girth assumption of Corollary \ref{C: final graphical small cancellation for Rips-Segev}. 

\begin{lemma}
 For all natural numbers $n>1$, there are finite connected graphs with vertex degree $2n$ which have arbitrarily large girth.
\end{lemma}
\begin{proof}
 Let $B$ the set of elements in the free group on n generators that are at distance at most $r$ from the identity in the Cayley graph of free group.  As the free group is residually finite, there is a normal finite index subgroup $N$ in the free group, so that $B\cap N$ is the identity. Moreover, the free group on $n$ generators is the fundamental group of a bouquet of $n$ circles. The covering space of a bouquet of $n$ circles corresponding to $N$ is $2n$-regular and has girth at least $2r$.
 \end{proof}
 The lemma gives a graph $\Phi$ with vertex degree $8$ and $\girth(\Phi)>41$. We enumerate the vertices of $\Phi$ by $1,2,\ldots,n$. Our $\Phi$ is a covering space of the bouquet of $4$ circles with positive orientation labeled by $x_1,x_2,x_3,x_4$.  Hence, for every vertex $i$ of $\Phi$ there are four vertices denoted by $l_{ij}$, $1\leq j\leq 4$, and four directed edges $x_{ij}:=({l_{ij}},i)$ labeled by $x_j$. For every vertex $i$ there are four vertices denoted by $k_{ij}$, $1\leq j \leq 4$, and four directed edges $y_{ij}:=(i,{k_{ij}})$ labeled by $x_j$. We have that $y_{ij}=x_{k_{ij}j}$, cf. Figure~\ref{F: exug}. 
 
 We relabel $\Phi$ by words in $a^{\pm1}$ and $b^{\pm1}$ as follows. For all $i$ we set $L(x_{i1})=ba^{-I_{i1}}$, $L(x_{i2})=a^{C_{l_{i2}}}ba^{-I_{i2}}$, $L(y_{i3})=a^{O_{i1}}ba^{-C_{k_{i3}}}$ and $L(y_{i4})=a^{O_{i2}}b$. As $y_{ij}=x_{k_{ij}j}$ and $x_{ij}=y_{l_{ij}i}$, this yields a new labeling of~$\Phi$. Figure \ref{F: exug} shows the local picture of a graph $\Phi$ labeled with respect to suitable coefficients $(I,O,C)$. The original graphs used by Rips-Segev are such graphs $\Phi$.

 \begin{figure}
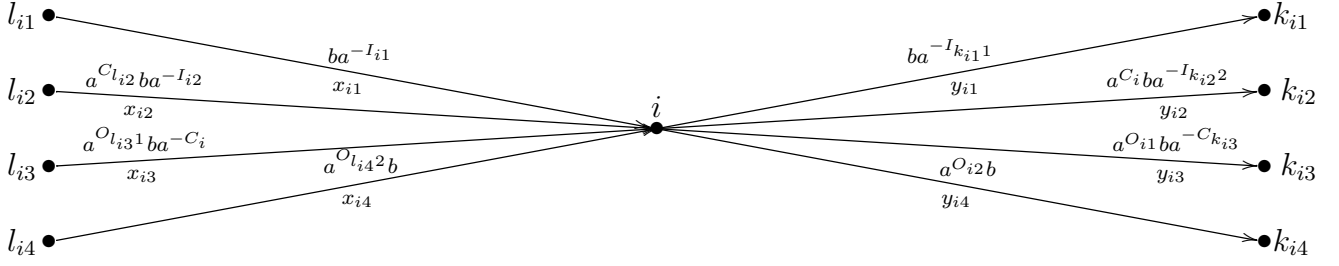
 
\[\xy
{\ar^{ba^{-I_{i1}}}_{x_{i1}} (-80,15)*{\bullet}; (0,0)*{\bullet}}; {\ar^{ba^{-I_{k_{i1}1}}}_{y_{i1}}  (0,0); (80,15)*{\bullet}}; (0,3)*{i}; (-83.5,15)*{{l_{i1}}}; (83.5,15)*{{k_{i1}}};
{\ar^<<<<<<<<<<{a^{C_{l_{i2}}}ba^{-I_{i2}}}_<<<<<<<<<<{x_{i2}} (-80,5)*{\bullet}; (0,0)}; {\ar^>>>>>>>>>>{a^{C_i}ba^{-I_{k_{i2}2}}}_>>>>>>>>>>{y_{i2}}  (0,0); (80,5)*{\bullet}}; (-83.5,5)*{{l_{i2}}}; (84.5,5)*{{k_{i2}}};
{\ar^<<<<<<<<<<<{a^{O_{l_{i3}1}}ba^{-C_i}}_<<<<<<<<<<{x_{i3}} (-80,-5)*{\bullet}; (0,0)}; {\ar^>>>>>>>>>>{a^{O_{i1}}ba^{-C_{k_{i3}}}}_>>>>>>>>>>{y_{i3}} (0,0); (80,-5)*{\bullet}}; (-83.5,-5)*{{l_{i3}}}; (84.5,-5)*{{k_{i3}}};
{\ar^<<<<<<<<<<<<<<<<<<<<<<<<<<<<<<<<<<<{a^{O_{l_{i4}2}}b}_{x_{i4}} (-80,-15)*{\bullet}; (0,0)}; {\ar^{a^{O_{i2}}b}_{y_{i4}} (0,0); (80,-15)*{\bullet}}; (-83.5,-15)*{{l_{i4}}}; (83.5,-15)*{{k_{i4}}};
 \endxy
\] \caption{Local picture of an underlying graph at the vertex $1$, cf. \cite{rips_torsion-free_1987}*{Figure p.119}} \label{F: exug} 
\end{figure}

 By construction we now have the following observation.
 \begin{prop}The $\{a,b\}$--reduction of $\Phi$ labeled $L$ is a generalized Rips-Segev graph with coefficients $(I,O,C)$ and underlying graph $\Phi$.\end{prop}

  Hence, we obtain first explicit examples of generalized Rips-Segev graphs $\Gamma$ that satisfy, by Corollary \ref{C: final graphical small cancellation for Rips-Segev} the $Gr_*'(1/8)$--condition. We have infinitely many such $\Gamma$ with different coefficients $(I,O,C)$.
 



\smallskip

The construction of underlying graphs $\Phi$ we have explained in this Section can further be extended, and gives more examples of generalized Rips-Segev graphs:
\begin{itemize}
\item  First, we can allow other and non-uniform values for $s_i$ and $t_i$. This corresponds to graphs $\Phi$ with non-uniform vertex degree. The Rips-Segev condition then states that all distances between all pairs of the distinguished vertices $v_{i0},$ $v_{iI_{ij}},$ $v_{iO_{ik}}$ and $v_{iC_i}$ on the $a$-lines are distinct among each other. Examples of such coefficients clearly exist: one can reorder the numbers used in Example~\ref{E:coefM=0}. Graphs  $\Gamma$ with such coefficients have  untypical $a$-lines, see for example Figures~\ref{F: extremal case 1} and~\ref{F: extremal case 2}. We still have that $\Lambda(\Gamma)\leqslant 3$.

Clearly, it follows from our proof that graphs $\Phi$ with non-uniform vertex degree and arbitrary large girth are available. Similar to our above explanation, such graphs can be labeled by words in $a^{\pm 1},b^{\pm 1}$ so that the $\{a,b\}$--reduction yields graphs $\Gamma$ with possibly untypical $a$-lines. As before, if $\Phi_{\Gamma}$ has girth larger than $41$, the labeling of $\Gamma$ satisfies the $Gr_*'(1/8)$--condition.

\item 
Let $M$ be a non-zero natural number. Let us  allow $M$ of the numbers in the Rips-Segev condition to be equal. Then the maximal piece length in a corresponding $\Gamma$ is bounded by $2M+3$. Given a more technical large girth assumption, stating that the girth of $\Phi$ is `large with respect to $M$', we can conclude that the labeling of  $\Gamma$ satisfies the $Gr_*'(1/8)$--condition.

\item If $F=G_1*\cdots *G_d$, $d\geqslant 2$ and all $G_i$ are torsion-free, choose $1\not = a\in G_1$ and $1\not= b_{i-1}\in G_i$. Let $A$ be given as above and let $B:=\{1,a,b_1,\ldots,b_{d-1},ab_1,\ldots,a{b_{d-1}}\}$. 
 We can then extend our constructions to such sets $A$ and $B$ and produce more generalized Rips-Segev graphs. 

\end{itemize}

\subsection{Finite and infinite families of generalized Rips-Segev graphs} \label{S: families of Rips-Segev graphs}

Let $(\Gamma_l)_l$ be a (possibly infinite) family of generalized Rips-Segev graphs for sets $A_l$, $B_l$ and let $\overrightarrow{\Gamma}:=\bigsqcup_l \Gamma_l$. If the labeling of $\overrightarrow{\Gamma}$ satisfies the $Gr_*'(1/8)$--condition, a corresponding group presentation is called \emph{generalized Rips-Segev presentation}. The proof of Proposition \ref{P: small cancellation on RSgraphs} applies to such an infinite family of graphs so that each pair $A_l,B_l$ does not have a unique product in $G(\bigsqcup_l \Gamma_l)$.  In analogy with Section \ref{S: RipsSegev condition} we now describe conditions on the family $(\Gamma_l)_l$ that are sufficient to conclude. Section \ref{S: explicit construction} can then be used to give explicit examples of such families.  

\smallskip

Each generalized Rips-Segev graph $\Gamma_l$ has $K_l$-many $a$-lines and coefficients denoted by  $(I,O,C)_l=\bigcup_{1\leqslant i_l\leqslant K_l} (i_l,s_{i_l},t_{i_l},(I_{i_l},O_{i_l},C_{i_l}))$. Set $j_l:=i_l+K_1+\ldots+K_{l-1}$ and let us enumerate all the $a$-lines in~$\overrightarrow{\Gamma}$ so that the image of the $i_l$-th $a$-line of $\Gamma_l$ is the $j_l$-th $a$-line of $\overrightarrow{\Gamma}$. In addition, we replace each $i_l$ in $(i_l,s_{i_l},t_{i_l},(I_{i_l},O_{i_l},C_{i_l}))$ by $j_l$. We get 
\[
\overrightarrow{(I,O,C)}= \overrightarrow{\bigcup_l} (I,O,C)_l:= \bigcup_l \bigcup_{1\leqslant i_l \leqslant K_l} (j_l,s_{j_l},t_{j_l},(I_{j_l},O_{j_l},C_{j_l})), 
 \]
 the set of coefficients associated to $(\Gamma_l)_l$.
 %
 %
 
 \smallskip
 
  The Rips-Segev condition of Section \ref{S: RipsSegev condition} applies to such finite or infinite $\overrightarrow{(I,O,C)}$. Under the Rips-Segev condition on $\overrightarrow{(I,O,C)}$, Lemma  \ref{L: maxplength} and Lemma \ref{L: mincycle} generalize as follows. 
\begin{lemma}
Under the Rips-Segev condition, for all graphical pieces $p$ in $\overrightarrow{\Gamma}$, we have that $\lvert \omega(p) \rvert_* \leqslant 3$.
\end{lemma}
\begin{lemma}
Under the Rips-Segev condition, we have that  
 $
 \gamma(\overrightarrow{\Gamma})\geqslant \min_l\{\girth (\Phi_{\Gamma_l}) +2\}.
 $
\end{lemma}
 
 We conclude the analogue for families of Corollary \ref{C: final graphical small cancellation for Rips-Segev}.   Example \ref{E:coefM=0} gives uncountably many infinite  sets of tuples with the Rips-Segev condition (there are uncountably many maps $\varphi: \mathbb{N} \to \mathbb{N}$). 
Using these sets of tuples, our constructions give uncountably many families $(\Gamma_l)_l$ such that $\overrightarrow{(I,O,C)}$ satisfies the Rips-Segev condition. 
  Section \ref{S: explicit construction} can be used to construct such $(\Gamma_l)_l$ with $\girth(\Phi_{\Gamma_l})>41$ for each $l$, so that $\overrightarrow{\Gamma}$ satisfies the $Gr_*'(1/8)$--condition.   
 Let us summarize the content of this section:
 \begin{theorem}\label{T: hyperbolic without up}
  There are uncountably many generalized Rips-Segev  presentations. The corresponding groups are torsion-free  and without the unique product property. If $G_1, \ldots, G_d$ are Gromov hyperbolic and the presentation is finite, then the corresponding groups are Gromov hyperbolic.
 \end{theorem}
 
 We have obtained the first examples of Gromov hyperbolic groups without the unique product property. We further study the, possibly infinite, generalized Rips-Segev presentations in Section \ref{S: number generalized Rips-Segev groups} below. 
Let us conclude this section with some conclusion on the original Rips-Segev construction.

\subsection{Rips-Segev groups revisited}
\label{S: Rips-Segev revisited 2}

 As mentioned above, Rips-Segev defined their graphical group presentations using graphs $\Phi$. We have explained a variant of their construction of such graphs. 
The coefficients used by Rips-Segev in their original construction satisfy our Rips-Segev condition, and the $\{a,b\}$--reduction of original Rips-Segev's $\Phi$ are generalized Rips-Segev graphs, denoted by  $\Gamma_{\Phi}$.  
 Our arguments show that $\Gamma_\Phi$ satisfies the $Gr_*'(1/8)$--condition. Rips-Segev's original groups coincide with the groups $G(\Gamma_{\Phi})$. We conclude that Rips-Segev's original groups \cite{rips_torsion-free_1987} are indeed torsion-free, and we provided a full proof of this fact. In addition, we showed that $A$ and $B$ inject in $G(\Gamma_{\phi}$. This completes the proof of the claim in \cite{rips_torsion-free_1987}*{p. 117}.  Moreover, we have the following new result. 
 
 \begin{theorem} \label{T: RS are hyperbolic}
 Rips-Segev's original groups \cite{rips_torsion-free_1987} are non-elementary hyperbolic.
 \end{theorem}
In contrast,  all the other known examples of the non-unique product groups are not hyperbolic, see \cite{promislow_simple_1988,carter_new_2013}. 
Finally, let us emphasize, again, that our generalized Rips-Segev graphs allow to study more presentations of groups without the unique product property.  In fact, we have the following. 
 \begin{remark} The class of our generalized Rips-Segev graphs is the most general in the following sense. Suppose our sets $A$ and $B$ embed in a group $G$ such that $A$ and $B$ do not have a unique product in $G$, then there is a generalized Rips-Segev graph $\Gamma$ for $A$ and $B$, and a presentation of $G$ including relators represented by the labels on the cycles of $\Gamma$.
 \end{remark}

This remark is of importance as such presentations  are not generic. We show this in the next section.


\section{The size of the class of generalized Rips-Segev groups}

We are interested in the size of the class of generalized Rips-Segev groups. In most generality, we ask whether or not two non-equivalent and non-isomorphic Rips-Segev graphs define isomorphic groups. The answer to this problem is unknown. We give a partial answer. In particular, we give a family of groups that are non-isomorphic as marked groups. We use this to construct an uncountable family of torsion-free groups without the unique product property. Then we show that finite generalized Rips-Segev presentations are not generic in the fundamental models of random finitely presented groups. 

\subsection{The number of generalized Rips-Segev groups} \label{S: number generalized Rips-Segev groups} 
We apply Lemma \ref{L: graphical small cancellation lemma} to groups given by generalized Rips-Segev presentations defined over a suitable family of generalized Rips-Segev graphs.

Let $\Gamma$ and $\Gamma'$ denote generalized Rips-Segev graphs whose labeling satisfies the $Gr_*'(1/8)$--condition. Let $D$ be a minimal van Kampen diagram over $R$ given by  $\Gamma$. We will assume that $D$ has non originating edges. Otherwise, we replace $D$ with our diagram $\widetilde{D}$.

\begin{lemma} \label{L: number} In $D$, there is at least one exterior face $\Pi$ with the following property. There is a non-zero number $P$ such that  $b^{\varepsilon_1}a^{P}b^{\varepsilon_2}$ is represented by the label of a simple path in $\partial_{ext}\Pi$.
 \end{lemma}
 
\begin{proof}
 By construction, there are pieces of length $3$: the labels of these pieces represent the element given as $a^{P_0}b^{\varepsilon}a^{P_1}$ in normal form. Hence, the  minimal cycle length $\gamma\geqslant 18$. If $D$ has more than one face, one exterior face $\Pi$ satisfies $|\partial_{ext}\Pi|_*>\frac{|\partial\Pi|_*}{2}$ by Lemma \ref{L: graphical small cancellation lemma}.
     If $|\partial_{int}\Pi|_* \geqslant 6$, then, by the above inequality, we have $|\partial_{ext}\Pi|_* \geqslant 4$. This implies our claim in this case.
    If $|\partial_{int}\Pi|_* <6$, as $\gamma\geqslant 18$ we have that $|\partial_{ext}\Pi |_*\geqslant 12$. Hence, the claim holds.  
\end{proof}


\begin{prop}\label{P: marked groups}
Let $\Gamma$ and $\Gamma'$ be two generalized Rips-Segev graphs with coefficients $(I,O,C)$ and $(I',O',C')$ such that  $(I,O,C) \overrightarrow{\bigcup} (I',O',C')$ (as defined in Section \ref{S: families of Rips-Segev graphs}) satisfies the Rips-Segev condition. Then the identity map $a\mapsto a$, $b\mapsto b$, does not induce an isomorphism between $G(\Gamma)$ and $G(\Gamma')$.
\end{prop}
In other words, the groups $G(\Gamma)$ and $G(\Gamma')$ are not isomorphic as marked groups.
\begin{proof}
 Let us assume that the identity induces an isomorphism $G(\Gamma)\to G(\Gamma')$. Let 
 \[
r=a^{P_0}b^{\varepsilon_1}a^{P_1}b^{\varepsilon_2}\cdots a^{P_{l-1}}b^{\varepsilon_l}  
 \]
be represented by the label of a simple cycle of $\Gamma$. Let $D$ be a minimal van-Kampen diagram for $r$ over the relators of $G(\Gamma')$. Then $D$ has more than one face, otherwise $r$ can be represented by the label of a cycle of $\Gamma'$, a contradiction. The non-zero number $Q$ given by the above lemma equals to an exponent $P_i$ in the representation of $r$. Remark \ref{R: RScondition 1} implies that both $Q$ and $P_i$ are among the numbers in the Rips-Segev condition. This is a contradiction to the choice of $\Gamma$ and $\Gamma'$.
\end{proof}

We can easily extend the above argument to prove the following observation.

\begin{remark}
 Let $\Gamma$ and $\Gamma'$ be as in Proposition \ref{P: marked groups} above. In addition,  suppose the numbers in the Rips-Segev condition do not differ by $\pm1$, and do not equal to $1$ or $2$.
 \begin{itemize}
  \item  Single elementary Nielsen equivalences do not induce isomorphisms of $G(\Gamma)$ and $G(\Gamma')$.  
  \item  If no number in the Rips-Segev condition on $(I,O,C)$ is a multiple of a number in the Rips-Segev condition on $(I',O',C')$, then the map $a\mapsto ba^Pb^{-1}$, $b\mapsto b$ is no isomorphism  of $G(\Gamma)$ and $G(\Gamma')$ for any $P\in \mathbb{Z}$.
  \end{itemize}
\end{remark}

Let $G=F/\langle\langle R\rangle\rangle$ and $W\subseteq R$.  \emph{The non-unique product property of $A$ and $B$ in $G$ is a consequence of $W$} if all words $u:=xy(x'y')^{-1}$, $x,x'\in A$, $y,y'\in B$ are  in the normal closure of $W$ in~$F$. 
\begin{cor}
Let $(\Gamma_l)_l$ be a family of generalized Rips-Segev graphs for $A_l$ and $B_l$. Under the Rips-Segev condition on $\overrightarrow{(I,O,C)}$, for all $j$, the non-unique product property of $A_j$ and $B_j$ is \emph{not} a consequence of the relators representing the labels of the cycles of $\bigsqcup_{l, l\not=j}\Gamma_l$.
\end{cor}
This follows from Lemma \ref{L: number} as in the proof of Proposition \ref{P: marked groups}.

\begin{theorem}\label{infpresgrwup}
There are uncountably many non-isomorphic torsion-free groups without the unique product property.  
\end{theorem}

Such a countable family has recently been constructed (in a quite technical way) in \cite{carter_new_2013} via the Passman-Promislow example of a torsion-free group without the unique product property \cites{promislow_simple_1988}. 

\begin{proof}
 We adapt a standard argument. Let $(\Gamma_l)_{l\in \mathbb{Z}}$ be an infinite family of generalized Rips-Segev graphs such that the labeling of $\overrightarrow{\Gamma}:=\bigsqcup_l \Gamma_l$ satisfies the $Gr_*'(1/8)$--condition. The corresponding group $G\left(\Omega \right)$ is torsion-free and without the unique product property. For every subset $I\subseteq \mathbb{Z}$, we let  $\overrightarrow{\Gamma}_I:=\bigsqcup_{i\in {I}}\Gamma_i$. The groups $G_I:=G(\overrightarrow{\Gamma}_I)$ are torsion-free and without the unique product property. Taking into account Proposition \ref{P: marked groups} and its proof  
 , we obtain uncountably many groups $G_I$ that are pairwise non-isomorphic as marked groups. As a finitely generated group has only countably many pairs of generators, there are uncountably many non-isomorphic $G_I$.
\end{proof}

An interesting open question is whether or not there are infinitely many non-isomorphic finitely presented Rips-Segev groups. In the next section, we show that the class of finitely presented Rips-Segev groups is small when considered within certain models of random finitely presented groups.

\subsection{Genericity and generalized Rips-Segev presentations}

The following lemma is useful to prove that the class of finitely presented Rips-Segev groups is not generic in the Gromov graphical model as well as in the Arzhantseva-Ol'shanskii few relator model of random finitely presented groups.

\begin{lemma}\label{reduced word 2} A generalized Rips-Segev graph with underlying graph of girth $g$ has a reduced path with label  $$a^{P_0}b^{\varepsilon}a^{P_1}b^{\varepsilon}a^{P_2} \ldots a^{P_{g-1}}b^{\varepsilon}a^{P_g},$$ where $\varepsilon=\pm 1$ and $P_i\not=0$ for all $0<i<g$. 
\end{lemma}

\begin{proof} We indicate how to find such a path. 
Start at $v_{i0}$ or $v_{iC_i}$. Follow the $b$-edge pointing from the $a$-line $i$ to the vertex $v_{j,O}$ (or possibly $v_{j0}$, $v_{jC_j}$). Then go along the $a$-line $j$ towards an (the other) end of $j$. We reach the $a$-line $l$ at $v_{lO}$ (or possibly $v_{l0}$, $v_{lC_l}$). Then go to an (the other) end $v_{l0}$, $v_{lC_l}$. As the girth of the underlying graph is $g$, we can go along at least $g$ $b$-edges before coming back to the $a$-line~$i$. 
\end{proof}

\begin{figure}
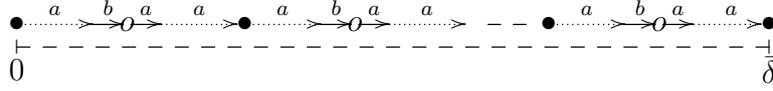

 \[
\xy
{\ar@{.>}^a (0,0)*{\bullet};(10,0)};{\ar^{b} (10,0);(14,0)}; (14.5,0)*{o}; 
 {\ar^a (15,0);(19,0)};{\ar@{.>}^a (19,0);(29,0)}; 
{\ar@{.>}^a (30,0)*{\bullet};(40,0)};{\ar^b (40,0);(44,0)}; (44.5,0)*{o}; 
 {\ar^a (45,0);(49,0)};{\ar@{.>}^a (49,0);(59,0)};
{(62,0)*{}; (67,0)**\dir{--}}; 
{\ar@{.>}^a (70,0)*{\bullet};(80,0)};{\ar^b (80,0);(84,0)}; (84.5,0)*{o}; 
 {\ar^a (85,0);(89,0)};{\ar@{.>}^a (89,0);(99,0)*{\bullet}}; {\ar@{|--|} (0,-3);(99,-3)}; (99,-6)*{\bar{\delta}}; (0,-6)*{0};
\endxy
\] \caption{$w=a^{P_0}b^{\varepsilon}a^{P_1}b^{\varepsilon}a^{P_2} \ldots a^{P_{k-1}}b^{\varepsilon}a^{P_k}$}
\label{specific word}
\end{figure}

If $\Delta$ is a graph, then let $\Delta^j$ denote the \emph{$j$-subdivision of $\Delta$}. This is the graph obtained by replacing each edge of $\Delta$ with $j$ edges. 

\begin{theorem}\label{g1} Generalized Rips-Segev presentations are not generic in Gromov's graphical model \cites{gromov_random_2003,ollivier_kazhdan_2007} of finitely presented random groups:

 For any $v\in \mathbb{N}$ and for any $C\geqslant 1$, there exists an integer $j_0$ such that for any $j\geqslant j_0$, for any family of graphs $\Delta=(\Delta_i)$ of girth $\delta=(\delta_i)$ satisfying the conditions 
\begin{enumerate} \item Each vertex of $\Delta$ is of valency at most $v$; \item $\diam(\Delta_i)\leqslant C\delta_i \text{ for all } i$,                        \end{enumerate}
  with probability tending to $1$ with exponential asymptotics as $\delta\to \infty$, the folded graph $\overline{\Delta^j}$ obtained by a random labeling of $\Delta^j$ contains no generalized Rips-Segev graph as a subgraph.
\end{theorem}
\begin{proof} For any small $\beta >0$, there is a number $j_0$ such that, for all $j\geqslant j_0$, the girth of the folded graph $\overline{\Delta^j}$ is at least $\bar{\delta}=(\eta-\beta)\delta(\Delta)$. Here $0<\eta<1$ is the gross cogrowth of a finitely generated free group  \cite{ollivier_kazhdan_2007}*{Proposition 7.8}.\footnote{Note however that we consider a ``two generator model'' for quotients of the free product $F=G_1*G_2$, in the sense that the words in $R$ are generated by concatenating the words $a^{\pm1}\in G_1$ and $b^{\pm 1}\in G_2$. In particular,  we consider random labellings of $\Delta^j$ by the words $a\in G_1$ and $b\in G_2$, and therefore random quotients of $F$.}

We show that the folded graph $\overline{\Delta^j}$ obtained by a random labeling of $\Delta^j$ contains no generalized Rips-Segev graph with exponential asymptotics as $\delta\to \infty$. As usual, we denote by $\Delta $ also a member of the family $(\Delta_i)$. If $\overline{\Delta^j}$ contains a generalized Rips-Segev graph, by Lemma \ref{reduced word 2} there is a path $p$ of length $\bar{\delta}$ in $\overline{\Delta^j}$, bearing a word of type $$w=a^{P_0}b^{\varepsilon}a^{P_1}b^{\varepsilon}a^{P_2} \ldots a^{P_{k-1}}b^{\varepsilon}a^{P_k},$$ where $\varepsilon=\pm 1$ and $P_i\not=0$ for all $0<i<k$, see Figure~\ref{specific word}.

We have to estimate the number of such paths in $\overline{\Delta^j}$.

First, we show that the number of the words $w$ read on a path of length $\bar{\delta}$ is at most $2^{\bar{\delta}+2}$. In fact, there are $2$ possibilities for $\varepsilon$.  There are at most $\sum_{k=0}^{\bar{\delta}/2}\binom{\bar{\delta}}{k}$ possibilities for the vertices $o$ in between $b$ and $a$. Such a choice determines the lengths of the dotted $a$-paths. For every vertex $o$ we have to choose an orientation for $a$. This gives at most $$2 \sum_{k=0}^{\bar{\delta}/2} \binom{\bar{\delta}}{k} 2^k\leqslant 2\cdot 2^{\bar{\delta}/2}  \cdot \sum_{k=0}^{\bar{\delta}/2} \binom{\bar{\delta}}{k}\leqslant 2 \cdot 2^{\bar{\delta}+1}$$ possibilities for $w$ read on a path of length $\bar{\delta}$ in $\overline{\Delta^j}$. 
 
Our path is the folding of a path of length at least $\bar{\delta}=(\eta-\beta)\delta j$ in $\Delta^j$.  The number of paths of length $\geqslant (\eta-\beta)\delta j$ in $\Delta^j$ is at most $ C\delta j^2v^{C\delta+C\delta}$. Thus there are at most $C\delta j^2v^{C\delta+C\delta}2^{\bar{\delta}+2}$ possibilities for the occurrence of an above path $p$ in $\overline{\Delta^j}$. 

The number of reduced words of length $\bar{\delta}=(\eta-\beta)\delta j$ is at least $3^{(\eta-\beta)\delta j}$. Thus the probability that a randomly labeled graph $\overline{\Delta^j}$ is a generalized Rips-Segev graph is bounded by $$C\delta j^2v^{2C\delta}\frac{2^{(\eta-\beta)\delta j}}{3^{(\eta-\beta)\delta j}}.$$ We choose $j$ so large that $v^{2C}\left(\frac{2}{3}\right) ^j<1$. As $\delta\to \infty$, the probability that $\overline{\Delta}$ is a generalized Rips-Segev graph tends to zero with exponential asymptotics.
\end{proof}

\begin{theorem}\label{g2} $\quad$ Generalized Rips-Segev presentations are not generic in Arzhantseva-Ol'shanskii's few-relator model \cite{arzhantseva_class_1996} of finitely presented random groups:

For all $n\in \mathbb{N}$, the probability that a randomly chosen group presentation among all group presentations $$\langle a,b \mid r_1,\ldots ,r_n\rangle, \text{ where } r_i \text{ is a cyclically reduced word in } a,b \text{ and } |r_i|\leqslant t,$$ is a generalized Rips-Segev presentation tends to zero with exponential asymptotics as $t\to \infty$. \end{theorem}

\begin{proof} We show that the number of words of type $$w=a^{P_0}b^{\varepsilon}a^{P_1}b^{\varepsilon}a^{P_2} \ldots a^{P_{k-1}}b^{\varepsilon}a^{P_k},$$ where $\varepsilon=\pm 1$, $P_i\not=0$ for all $0<i<k$ and length $\leqslant t$, is at most $4t\cdot 2^{t}$. In fact, there are $2$ possibilities for the orientation of $b$.  There are at most $$\sum_{l=0}^{t}\sum_{k=0}^{l/2}\binom{l}{k}\leqslant t\sum_{k=0}^{t/2} \binom{t}{k}$$ possibilities for the vertices $o$ in between $b$ and $a$. These choices include all such words with length between $0$ and $t$. Such a choice determines the lengths of the dotted $a$-paths. For every vertex $o$ we have to choose an orientation for $a$. This gives at most $$2 t\sum_{k=0}^{t/2} \binom{t}{k} 2^k\leqslant 2t\cdot 2^{t/2}  \cdot \sum_{k=0}^{t/2} \binom{t}{k}\leqslant 2 t\cdot 2^{t+1}$$ possibilities for such a word of length $\leqslant t$.

The probability that a random presentation in the above sense contains a relator $w$ is at most $$\frac{4t 2^{t}\cdot 2^{n-1}3^{(n-1)t}}{3^{n(t-1)}},$$ and tends to zero with exponential asymptotics as $t\to \infty$.
 \end{proof}
 
Our results suggest that unique product groups are generic among finitely presented groups. This would then imply that the Kaplansky zero divisor conjecture is generic among finitely presented groups.

\end{document}